\documentclass[11pt,final]{amsart}
\usepackage{amstext, amscd, amsmath, amssymb}
\usepackage{mathtools}

\numberwithin{equation}{section}
\usepackage[a4paper]{geometry}
\let\OLDthebibliography\thebibliography
\renewcommand\thebibliography[1]{
  \OLDthebibliography{#1}
  \setlength{\parskip}{0pt}
  \setlength{\itemsep}{2pt plus 0.5ex}
}

\makeatletter
\def\@cite#1#2{{\m@th\upshape\bfseries%
[{#1\if@tempswa{\m@th\upshape\mdseries, #2}\fi}]}}
\makeatother
\theoremstyle{plain}
\newtheorem{theorem}{Theorem}[section]
\newtheorem{corollary}[theorem]{Corollary}
\newtheorem{proposition}[theorem]{Proposition}
\newtheorem{lemma}[theorem]{Lemma}
\theoremstyle{definition}

\newtheorem{remark}[theorem]{Remark}

\theoremstyle{remark}

\mathtoolsset{centercolon}
  \newcommand{\A}{{\mathcal{A}}}
  \newcommand{\B}{{\mathcal{B}}}
  \newcommand{\C}{{\mathcal{C}}}
  \newcommand{\D}{{\mathcal{D}}}
  \newcommand{\E}{{\mathcal{E}}}
  \newcommand{\F}{{\mathcal{F}}}

  \newcommand{\I}{{\mathcal{I}}}
  
  \newcommand{\K}{{\mathcal{K}}}
\renewcommand{\L}{{\mathcal{L}}}
  \newcommand{\M}{{\mathcal{M}}}
  \newcommand{\N}{{\mathcal{N}}}
\renewcommand{\O}{{\mathcal{O}}}

\renewcommand{\S}{{\mathcal{S}}}
  \newcommand{\T}{{\mathcal{T}}}

  \newcommand{\X}{{\mathcal{X}}}
  \newcommand{\Y}{{\mathcal{Y}}}
  
\newcommand{\eps}{\varepsilon}
\def\al{\alpha}
\def\be{\beta}

\def\ga{\gamma}
\def\De{\Delta}
\def\de{\delta}
\def\ze{\zeta}
\def\io{\iota}
\def\ka{\kappa}
\def\la{\lambda}
\def\La{\Lambda}
\def\om{\omega}

\newcommand\vphi{\varphi}

\newcommand{\bC}{\mathbb{C}}

\newcommand{\bN}{\mathbb{N}}

\newcommand{\bZ}{\mathbb{Z}}
\newcommand{\bR}{\mathbb{R}}

\newcommand{\fg}{{\mathfrak{g}}}
\newcommand{\fH}{{\mathfrak{H}}}
\newcommand{\fh}{{\mathfrak{h}}}

\newcommand{\fs}{{\mathfrak{s}}}

\newcommand{\foral}{\text{ for all }}
\newcommand{\qand}{\quad\text{and}\quad}

\newcommand{\qiff}{\quad\text{if and only if}\quad}

\newcommand{\ca}{\mathrm{C}^*}

\newcommand{\cenv}{\mathrm{C}^*_{\textup{env}}}

\newcommand{\ol}{\overline}

\newcommand{\wh}{\widehat}

\newcommand{\ad}{\operatorname{ad}}

\newcommand{\Aut}{\operatorname{Aut}}

\newcommand{\cov}{\operatorname{c}}

\newcommand{\fock}{\operatorname{F}}

\newcommand{\id}{{\operatorname{id}}}

\newcommand{\mt}{\emptyset}

\newcommand{\scv}{\operatorname{sc}}

\newcommand{\supp}{\operatorname{supp}}

\newcommand{\sca}[1]{\left\langle#1\right\rangle} 

\begin{document}

\title[The reduced Hao--Ng isomorphism problem for product systems]{The reduced Hao--Ng isomorphism problem for product systems}

\author[I.A. Paraskevas]{Ioannis Apollon Paraskevas}
\address{Department of Mathematics\\ National and Kapodistrian University of Athens\\ Athens\\ 157 84\\ Greece}
\email{ioparask@math.uoa.gr}


\begin{abstract}
We prove that the reduced Hao--Ng isomorphism problem has an affirmative answer for every generalised gauge action of a locally compact Hausdorff group on a product system over a unital subsemigroup of a discrete group, without any assumptions on the product system or the subsemigroup.
We show that the Fock C*-algebra of the induced product system is canonically isomorphic to the reduced crossed product of the Fock C*-algebra, by identifying the Fock space of the induced product system with the reduced crossed product of the Fock space.
This also yields the corresponding completely isometric isomorphism for the tensor algebras.
Using the representations of the fixed point algebra that detect Sehnem's strong covariance relations, we show that the canonical representation of the induced product system in the reduced crossed product of the reduced strong covariant C*-algebra is strong covariant.
\end{abstract}

\maketitle

\section{Introduction}

\subsection{Product systems framework}

Product systems over a unital subsemigroup $P$ of a discrete group $G$ provide a framework for modelling irreversible dynamics.
Their operator algebras include those arising from subshifts, graphs and higher-rank graphs, topological graphs, and semigroup actions on C*-algebras.
They first appeared in the work of Arveson \cite{Arv89} and were later studied by Dinh \cite{Din91} for discrete subsemigroups of $\bR_+$.
Nica \cite{Nic92} considered semigroups with a quasi-lattice order, where the left regular representation has particularly useful properties.
The relations coming from the Fock representation, now known as Nica covariance, give a more manageable class of representations.
Nica's setting may be viewed as a product system with one-dimensional fibres.

Pimsner \cite{Pim97} developed the theory for $P=\bZ_+$, where a product system is determined by a single C*-correspondence.
This construction includes graph C*-algebras, crossed products by $\bZ$, and algebras associated with partial dynamical systems; see for example \cite{Kat03}.
Motivated by \cite{Nic92,Pim97} and his work with Raeburn \cite{FR98}, Fowler \cite{Fow02} initiated the study of product systems over quasi-lattice ordered semigroups.
Brownlowe, Larsen and Stammeier \cite{BLS18} later considered product systems arising from dynamical systems over right LCM semigroups.
Kwa\'{s}niewski and Larsen \cite{KL19a,KL19b} extended this theory to general product systems over right LCM semigroups, which were further studied by Dor-On, Kakariadis, Katsoulis, Laca and Li \cite{DKKLL22}, and Kakariadis, Katsoulis, Laca and Li \cite{KKLL23}.
Beyond this setting, Wick ordering and Nica covariance are no longer available.
Nevertheless, Cuntz, Deninger and Laca \cite{CDL13} showed that the Fock representation can still be used as the model for covariance.

Coactions by the ambient group $G$ have played a central role in the theory of product systems, since they allow one to use tools from the theory of Fell bundles.
Dor-On, Kakariadis, Katsoulis, Laca and Li \cite{DKKLL22} observed that for a product system $X$ the normal coaction by $G$ on the Fock C*-algebra $\T_\la(X)$ gives a canonical Fell bundle, denoted here by $\F\C X$.
Its reduced C*-algebra is $\T_\la(X)$, while its full C*-algebra $\T_{\cov}^{\fock}(X)$ need not coincide with the universal Toeplitz algebra $\T(X)$ of all representations of $X$.
A representation of $X$ is called Fock covariant when it factors through $\T_{\cov}^{\fock}(X)$.
It was shown in \cite{DKKLL22} that, for compactly aligned product systems over right LCM semigroups, Fock covariance agrees with Nica covariance.
In this case, $\T_{\cov}^{\fock}(X)$ is the universal Nica covariant C*-algebra $\N\T(X)$.

Li \cite{Li12}, building on ideas from \cite{CDL13}, introduced relations that record not only the semigroup multiplication but also the principal right ideals and their intersections, known as the constructible right ideals.
These relations are modelled by the Fock representation and were studied further by Kakariadis, Katsoulis, Laca and Li \cite{KKLL22}, also in connection with the inverse semigroup models of Norling \cite{Nor14}.
Laca and Sehnem \cite{LS22} gave a complete description of the relations that make an isometric representation of a semigroup Fock covariant.
In \cite{KP25}, Kakariadis and the author gave a characterisation of injective $G$-equivariant Fock covariant representations of a product system without any assumptions on the semigroup or the product system.

A long-standing problem has been whether there exists a boundary quotient of $\T(X)$ for the injective $G$-equivariant Fock covariant representations of $X$, and how this quotient relates to the C*-envelope of the tensor algebra $\T_\la(X)^+$, in analogy with the Cuntz--Pimsner algebra $\O_X$ of Katsura \cite{Kat04} for $P=\bZ_+$.
Carlsen, Larsen, Sims and Vittadello \cite{CLSV11} proved the existence of such a boundary quotient in the case of compactly aligned product systems over quasi-lattice ordered pairs $(G,P)$ and under some assumptions that are satisfied by many important cases, such as $(\bZ^d,\bZ_+^d)$.
The universal construction for general product systems was given by Sehnem \cite{Seh19}, who introduced the quotient $A\times_X P$.
Here $A$ denotes the coefficient algebra of $X$, i.e., the unit fibre $A:= X_{e_P}$.
The key feature of $A\times_X P$ is that $A$ embeds in $A\times_X P$, and for a $*$-representation of $A\times_X P$, injectivity on $A$ forces injectivity on the fixed point algebra of $A\times_X P$.
Dor-On and Katsoulis \cite{DK20} then settled the problem for compactly aligned product systems over abelian lattice ordered groups by showing that the terminal object for injective $G$-equivariant Fock covariant representations exists and coincides with the C*-envelope of the tensor algebra, and with $A\times_X P$.
In \cite{DKKLL22}, Dor-On, Kakariadis, Katsoulis, Laca and Li identified $A\times_X P$ with the full C*-algebra of the strong covariant bundle, denoted by $\S\C X$.
Furthermore, in the compactly aligned right LCM case, they proved that the terminal object is canonically $*$-isomorphic to both the reduced C*-algebra of this Fell bundle, denoted by $A\times_{X,\la}P$, and the coaction C*-envelope.
Note that in the abelian lattice ordered case the full and reduced C*-algebras of $\S\C X$ coincide.
Finally, Sehnem \cite{Seh22} proved that completely isometric representations of $\T_\la(X)^+$ automatically admit a conditional expectation.
This led to the identification of $A\times_{X,\la}P$ with $\cenv(\T_\la(X)^+)$, which resolved the problem in general and unified all previous results.

\subsection{The reduced Hao--Ng isomorphism problem}

Suppose now that a locally compact Hausdorff group $\fH$ acts on $\T_\la(X)$ by a generalised gauge action $\al\colon\fH\to\Aut(\T_\la(X))$, i.e., a point-norm continuous homomorphism such that $\al_\fh(\la_p(X_p))=\la_p(X_p)$ for every $p\in P$ and $\fh\in\fH$.
There is an induced product system
\[
X\rtimes_{\al,\la}\fH
=
\{X_p\rtimes_{\al,\la}\fH\}_{p\in P}
\]
over $P$, with coefficient algebra $A\rtimes_{\al,\la}\fH$, realised in the reduced crossed product $\T_\la(X)\rtimes_{\al,\la}\fH$.
The action $\al|_{\T_\la(X)^+}$ extends to an action $\dot\al$ on the C*-envelope of $\T_\la(X)^+$, which is canonically $*$-isomorphic to $A\times_{X,\la}P$.
The reduced Hao--Ng isomorphism problem then asks whether the reduced strong covariant functor commutes with the reduced crossed product functor, namely whether there is a canonical $*$-isomorphism
\[
(A\rtimes_{\al,\la}\fH)
\times_{X\rtimes_{\al,\la}\fH,\la}P
\ \simeq\
(A\times_{X,\la}P)
\rtimes_{\dot\al,\la}\fH.
\]

We begin by discussing the problem in the case $P=\bZ_+$, where the product system $X$ is determined by a single C*-correspondence, and the preceding question reduces to the original problem for Cuntz--Pimsner algebras as introduced by Hao and Ng \cite{HN08}.
In \cite{HN08} the corresponding $*$-isomorphism for actions of amenable locally compact Hausdorff groups is established.
As applications of their results, they recover previous results on Hilbert bimodules by Abadie \cite{Aba10}, and on graph C*-algebras by Kumjian and Pask \cite{KP99}.
Furthermore, Katsoulis \cite{Kat17} connected the Hao--Ng isomorphism problem with the work of Echterhoff, Kaliszewski, Quigg and Raeburn \cite{EKQR06} on imprimitivity theorems for C*-dynamical systems.

The amenability assumption was subsequently weakened in several directions.
In particular, B\'edos, Kaliszewski, Quigg and Robertson \cite{BKQR15} resolved the problem for actions of exact discrete groups.
In a series of works initiated by Katsoulis and Ramsey \cite{KR19} and followed by \cite{Kat17, KR21}, the authors developed a very effective non-selfadjoint approach to the reduced Hao--Ng isomorphism problem.
In \cite{KR21} it was proved that in the case of $P=\bZ_+$ and for any locally compact Hausdorff group $\fH$ there exists a canonical $*$-isomorphism
\[
\T_\la(X \rtimes_{\al, \la} \fH) \simeq \T_\la(X) \rtimes_{\al, \la} \fH
\]
that restricts to a completely isometric isomorphism
\[
\T_\la(X \rtimes_{\al, \la} \fH)^+ \simeq \T_\la(X)^+ \rtimes_{\al, \la} \fH.
\]
This was previously shown when $\fH$ is discrete in \cite{Kat17} and when $\fH$ is locally compact Hausdorff and abelian in \cite{KR19}.
By employing the identification of the C*-envelope of the tensor algebra of a C*-correspondence obtained by Katsoulis and Kribs in \cite{KK06}, the problem is reduced in the case $P=\bZ_+$ to determining whether there exists a canonical $*$-isomorphism
\begin{equation} \label{eq:cenvcp}
\cenv(\T_\la(X)^+ \rtimes_{\al, \la} \fH)\stackrel{?}{\simeq} \cenv(\T_\la(X)^+) \rtimes_{\dot{\al}, \la} \fH.
\end{equation}

The isomorphism \eqref{eq:cenvcp} was established in \cite{Kat17} for any operator algebra that admits a contractive approximate unit (in the place of $\T_\la(X)^+$) when $\fH$ is discrete.
In \cite{KP25} Kakariadis and the author removed the assumption of the contractive approximate unit.
In \cite{KR19} it was established for operator algebras that admit a contractive approximate unit and $\fH$ locally compact Hausdorff and abelian, and in \cite{KR21} it was established for a locally compact Hausdorff group $\fH$ and hyperrigid operator algebras (in the sense of \cite{Arv11}) that admit a contractive approximate unit.

In a recent work, Dor-On and Thompson \cite{DT26} proved \eqref{eq:cenvcp} for a locally compact Hausdorff group $\fH$ acting on an operator algebra $\A$ with a contractive approximate unit, under the assumption that the dynamical system satisfies a separation property, which they call property (S).
They showed that property (S) is satisfied by the dynamical systems on tensor algebras induced by generalised gauge actions on non-degenerate C*-correspondences, and thus resolved the reduced Hao--Ng isomorphism problem for all non-degenerate product systems over $P=\bZ_+$ and all locally compact Hausdorff groups $\fH$.

Mukohara and Suzuki \cite{MS26} resolved the reduced Hao--Ng isomorphism problem for all product systems over $P=\bZ_+$ and all locally compact Hausdorff groups $\fH$, using C*-algebraic techniques and without the assumption of non-degeneracy.
Their proof relies on two facts. Namely, the reduced crossed product functor preserves injectivity of $\fH$-equivariant $*$-homomorphisms, and the reduced crossed product of an essential $\fH$-invariant ideal is an essential ideal of the reduced crossed product.
Combining these with Katsura's gauge-invariant uniqueness theorem \cite{Kat04}, they obtained the result.

For generalised gauge actions beyond the amenable case, there is also a version of the Hao--Ng isomorphism problem for full crossed products. For $P = \bZ_+$, Katsoulis and Ramsey \cite{KR21} obtained a positive answer for row-finite graph correspondences and any locally compact Hausdorff group $\fH$. However, in contrast to the reduced Hao--Ng isomorphism problem, the full Hao--Ng isomorphism problem has a negative answer in general, even when $P = \bZ_+$. Counterexamples were obtained by Dor-On and Thompson \cite{DT26b} and by Mukohara and Suzuki \cite{MS26}.

We now move the discussion to product systems over more general semigroups.
The reduced Hao--Ng isomorphism problem was resolved for a compactly aligned non-degenerate product system over an abelian lattice ordered group $(G,P)$ by Dor-On and Katsoulis \cite{DK20} when $\fH$ is a discrete group, and subsequently by Katsoulis \cite{Kat20} when $\fH$ is a locally compact Hausdorff and abelian group.
It was then resolved by Dor-On, Kakariadis, Katsoulis, Laca and Li \cite{DKKLL22} for a compactly aligned non-degenerate product system over a right LCM semigroup $P$ with $\fH$ discrete.
Kakariadis and the author \cite{KP25} resolved the problem for any product system over a unital subsemigroup $P$ of a discrete group $G$ and $\fH$ discrete.
We note that in \cite{KP25} the product system $X$ is not assumed to be non-degenerate.

The use of the C*-envelope toolkit for the reduced Hao--Ng isomorphism problem has been a prominent example of the interaction between selfadjoint and non-selfadjoint operator algebras.
This approach---pioneered by Katsoulis and Ramsey \cite{KR19}, and followed by many researchers in numerous subsequent works \cite{DKKLL22, DK20, KP25, Kat17, Kat20, KR21}---demonstrates how co-universal representations in the category of non-selfadjoint operator algebras yield a canonical $*$-isomorphism, providing a solution to a selfadjoint problem. In \cite{MS26} and in the present work a C*-algebraic route is taken.

\subsection{Approach and main results}

The purpose of this work is to settle the reduced Hao--Ng isomorphism problem in full generality, i.e., for any product system $X$ over any unital subsemigroup $P$ of a discrete group $G$ and any locally compact Hausdorff group $\fH$.

The main approach in \cite{DKKLL22, DK20, Kat20} has been to use the independence condition for right LCM semigroups and compact alignment of the product system to show Nica covariance of the identity representation
\begin{equation}\label{eq:ident}
X \rtimes_{\al, \la} \fH \hookrightarrow \T_\la(X) \rtimes_{\al, \la} \fH,
\end{equation}
in order to obtain a canonical $*$-isomorphism
\begin{equation}\label{eq:fccp}
\T_\la(X \rtimes_{\al, \la} \fH) \simeq \T_\la(X) \rtimes_{\al, \la} \fH.
\end{equation}
From there it follows that
\begin{equation}\label{eq:tencp}
\T_\la(X \rtimes_{\al, \la} \fH)^+ \simeq \T_\la(X)^+ \rtimes_{\al, \la} \fH,
\end{equation}
and then an identification of the form \eqref{eq:cenvcp} concludes the proof.
Beyond the right LCM setting, independence and Nica covariance are no longer available.
In \cite{KP25}, Kakariadis and the author proved that \eqref{eq:tencp} holds for any product system $X$ when $\fH$ is discrete, by showing that the identity representation \eqref{eq:ident} is Fock covariant.
There, the discreteness of $\fH$ is used for \eqref{eq:cenvcp}, and also to obtain a faithful conditional expectation on the reduced crossed product $\T_\la(X)\rtimes_{\al,\la}\fH$, which, together with the characterisation of injective $G$-equivariant Fock covariant representations obtained in \cite{KP25}, yields Fock covariance of the identity representation.
For a non-discrete group $\fH$ this strategy meets two obstructions: such a conditional expectation is no longer available, and, in the absence of non-degeneracy, $\T_\la(X)^+$ need not admit a contractive approximate unit, which is assumed in the results on \eqref{eq:cenvcp} for non-discrete groups \cite{DT26, KR19, KR21}. 

In the first arXiv version of this work we followed the approach of \cite{KP25} to prove \eqref{eq:tencp} for any non-degenerate product system $X$ over any unital subsemigroup $P$ of a discrete group $G$ and any locally compact Hausdorff group $\fH$. This was achieved by using the characterisation of Fock covariance obtained in \cite{KP25}, together with the Plancherel weight and the theory of integrable coactions of \cite{Bus10, BM09}, in order to show that the identity representation \eqref{eq:ident} is Fock covariant. We then combined \eqref{eq:tencp} with \eqref{eq:cenvcp} from \cite[Theorem~4.4 of the third arXiv version]{DT26} to resolve the reduced Hao--Ng isomorphism problem. However, the authors of \cite{DT26} informed the author that a gap had been found in the proof of \cite[Theorem~4.4 of the third arXiv version]{DT26}, and hence at present there is no known proof of \eqref{eq:cenvcp} for arbitrary operator algebras and arbitrary locally compact Hausdorff groups $\fH$, without the additional separation assumption of property (S) described above. We have therefore revised our work so as to obtain the positive resolution of the reduced Hao--Ng isomorphism problem without appealing to \eqref{eq:cenvcp}.

In this version, we give a different approach, inspired by \cite{MS26}, which relies neither on the characterisation of Fock covariance of \cite{KP25} nor on the results of \cite{DT26}. In particular, this approach does not make use of any results concerning the isomorphism \eqref{eq:cenvcp}. Moreover, this approach yields the positive resolution of the reduced Hao--Ng isomorphism problem in full generality, i.e., also for possibly degenerate product systems, and furthermore, it yields the canonical $*$-isomorphism \eqref{eq:fccp} regarding the Fock C*-algebras and the canonical completely isometric isomorphism \eqref{eq:tencp} regarding the tensor algebras.

The Fock space $\F X$ is an $\fH$-equivariant Hilbert $A$-module with respect to an action $\ga$ induced by $\al$, and we show that its reduced crossed product $\F X\rtimes_{\ga,\la}\fH$ is canonically unitarily equivalent to the Fock space $\F(X\rtimes_{\al,\la}\fH)$ of the induced product system; see Lemma~\ref{L:Fockcp}.
By Proposition~\ref{P:modrep}, the reduced crossed product of an invariant C*-algebra of adjointable operators on $\F X$ acts faithfully on $\F X\rtimes_{\ga,\la}\fH$.
Applied to $\T_\la(X)$, this gives a faithful representation of $\T_\la(X)\rtimes_{\al,\la}\fH$ on $\F(X\rtimes_{\al,\la}\fH)$ that sends the generators to the corresponding creation operators. This also yields a simplified proof of the main results of the first arXiv version of this work, namely, that the identity representation \eqref{eq:ident} is Fock covariant and that the completely isometric isomorphism \eqref{eq:tencp} holds, and also extends them to the possibly degenerate product system setting.

\medskip

\noindent
\textbf{Theorem~A.}
(Theorem~\ref{T:Fock} and Corollary~\ref{C:tenscros})
{\it
Let $X$ be a product system over a unital subsemigroup $P$ of a discrete group $G$, and let $\al$ be a generalised gauge action of a locally compact Hausdorff group $\fH$ on $\T_\la(X)$.
Then there exists a canonical $*$-isomorphism
\[
\T_\la(X \rtimes_{\al, \la} \fH) \simeq \T_\la(X) \rtimes_{\al, \la} \fH.
\]
In particular, the identity representation
\[
\io^\rtimes \colon X\rtimes_{\al,\la}\fH
\to
\T_\la(X)\rtimes_{\al,\la}\fH
\]
is an injective Fock covariant representation that admits a normal coaction by $G$, and there exists a canonical completely isometric isomorphism
\[
\T_\la(X \rtimes_{\al, \la} \fH)^+ \simeq \T_\la(X)^+ \rtimes_{\al, \la} \fH.
\]
}

\medskip

Towards the resolution of the reduced Hao--Ng isomorphism problem, instead of passing through \eqref{eq:cenvcp}, we verify strong covariance directly.
The strong covariance relations of Sehnem \cite{Seh19} are detected by the representations of the fixed point algebra $[\T_\la(X)]_{e_P}$ of $\T_\la(X)$ on the submodules
\[
X_F=\sum^\oplus_{r\in P}X_r I_{r^{-1}(r\vee F)}\subseteq\F X,
\]
where $F$ runs over the finite subsets of $G$ and the $I_{r^{-1}(r\vee F)}$ are ideals of $A$ determined by the kernels of the left actions.
In Lemma~\ref{L:compare} we show that these ideals pass to the crossed product, i.e., $I_{r^{-1}(r\vee F)}\rtimes_{\al,\la}\fH$ is contained in the corresponding ideal of $A\rtimes_{\al,\la}\fH$ for the induced product system, and hence $X_F\rtimes_{\ga,\la}\fH$ is mapped, through the unitary obtained in Lemma~\ref{L:Fockcp}, into the corresponding submodule $(X\rtimes_{\al,\la}\fH)_F$ of $\F(X\rtimes_{\al,\la}\fH)$.
On the other hand, the canonical quotient map $q_{{\rm sc},\la}\colon\T_\la(X)\to A\times_{X,\la}P$ onto the C*-envelope of $\T_\la(X)^+$ factors through each of these representations on $[\T_\la(X)]_{e_P}$, and this persists after taking reduced crossed products; see Lemma~\ref{L:factor}.
By combining these results, in Proposition~\ref{P:scov} we prove that the representation
\[
(q_{{\rm sc},\la}\rtimes\fH)\circ\io^\rtimes\colon X\rtimes_{\al,\la}\fH\to(A\times_{X,\la}P)\rtimes_{\dot\al,\la}\fH
\]
is strong covariant.
Moreover, in Proposition~\ref{P:injnor} we show that it is also injective and admits a normal coaction by $G$. With these at hand, an application of the main results of \cite{Seh19}, together with some well-known results in Fell bundle theory, gives the positive resolution of the reduced Hao--Ng isomorphism problem in full generality.

\medskip

\noindent
\textbf{Theorem~B.}
(Theorem~\ref{T:HaoNg})
{\it
Let $X$ be a product system over a unital subsemigroup $P$ of a discrete group $G$, and let $\al$ be a generalised gauge action of a locally compact Hausdorff group $\fH$ on $\T_\la(X)$.
Then the reduced Hao--Ng isomorphism problem has an affirmative answer, i.e., there is a canonical $*$-isomorphism
\[
(A \rtimes_{\al, \la} \fH) \times_{X \rtimes_{\al, \la} \fH,\la} P
\simeq
(A \times_{X, \la} P) \rtimes_{\dot \al, \la} \fH,
\]
where $\dot{\al}$ is the induced action of $\fH$ on $A \times_{X,\la} P$.
}

\subsection{Contents}

The paper is organised as follows.
In Section~\ref{S:prelim} we recall the required material on reduced crossed products of C*-algebras and Hilbert modules, operator algebras and product systems, including the strong covariance relations of Sehnem, while we also fix notation.
In Section~\ref{S:Fock} we identify the Fock space of the induced product system with the reduced crossed product of the Fock space, and we prove Theorem~A.
In Section~\ref{S:HaoNg} we compare the strong covariance relations of $X$ and $X\rtimes_{\al,\la}\fH$, and we prove Theorem~B.

\subsection*{Acknowledgements}

The author is grateful to Evgenios Kakariadis for fruitful discussions, and for his careful reading of and comments on a draft of this manuscript. The author is also grateful to Michalis Anoussis, Alexandros Chatzinikolaou and Aristides Katavolos for their comments and suggestions on a draft of this manuscript. The author is grateful to Adam Dor-On and Ian Thompson for communicating a gap in the proof of Theorem~4.4 of the third arXiv version of \cite{DT26}, which was used in the first arXiv version of this work, and for sharing their suggestions and comments on how to retain Theorem~B.

This research work was supported by the Hellenic Foundation for Research and Innovation (HFRI) under the 5th Call for HFRI PhD Fellowships (Fellowship Number: 19145).
This material is based upon work supported by the Swedish Research Council under grant no.\ 2021-06594 while the author was in residence at Institut Mittag-Leffler in Djursholm, Sweden, during the spring semester of 2026. The author would like to thank Institut Mittag-Leffler and the organisers of the programme ``Operator Algebras and Quantum Information'' for their hospitality.

\subsection*{AI Statement}

During the preparation of the present version of this manuscript, the author used Anthropic's Claude Opus 5.5 for proofreading, language editing and for checking the consistency of notation and cross-references.
It also pointed the author to \cite{BEW18}, which is used in Section~\ref{Ss:Cross}, and suggested direct proofs of Propositions~\ref{P:modsub} and~\ref{P:modrep}; the proofs given here are based on these suggestions.
These results are likely known to the experts, but the author could not locate a direct proof in the literature.
All mathematical statements, proofs and references were checked by the author, who takes full responsibility for the content of this paper.

\section{Preliminaries}\label{S:prelim}

We write $\otimes$ for the minimal tensor product between C*-algebras.
A map between algebras with given sets of generators is called canonical if it preserves generators of the same index.
If $\X$ is a subset of a normed linear space $\Y$, then we write $[\X]$ for the closed linear span of $\X$ in $\Y$. All ideals of C*-algebras considered in this work are norm-closed and two-sided ideals. For a subset $\X$ of a C*-algebra $\C$ we write $\sca{\X}$ for the ideal generated by $\X$ in $\C$.
To avoid any possible confusion, we always denote discrete groups by $G$ and locally compact Hausdorff groups by $\fH$.
When we consider a unital subsemigroup $P$ of a group $G$ we always mean that $e_P=e_G$.

\subsection{Hilbert modules, Fell bundles and coactions}\label{Ss:Fell}

We will be using the theory of C*-correspondences freely.
The reader may refer to \cite{Lan95, MT05} for the standard theory of Hilbert C*-modules.
For Hilbert C*-modules $E$ and $F$ over the same C*-algebra, we write $\L(E,F)$ for the adjointable operators from $E$ to $F$, and $\K(E,F)$ for the generalised compact operators, i.e., the closed linear span of the operators $\theta_{\eta,\xi}\colon\ze\mapsto\eta\sca{\xi,\ze}$ for $\eta\in F$ and $\xi\in E$.
We write $\L(E):=\L(E,E)$ and $\K(E):=\K(E,E)$, and we recall that $\L(E)\simeq\M(\K(E))$, where $\M(\cdot)$ denotes the multiplier algebra.

We will also be using the theory of Fell bundles over discrete groups and of coactions by discrete groups freely.
The reader may refer to \cite{Exe17} for the standard theory of Fell bundles and their C*-algebras, and to \cite{Qui96} for coactions by discrete groups and their connection with Fell bundles.
We follow the notation and terminology of \cite[Section~2.1]{KP25}, and we refer the reader there for the relevant definitions and further details.
For a Fell bundle $\B=\{\B_g\}_{g\in G}$ over a discrete group $G$, we write $\ca(\B)$ for the full C*-algebra of $\B$, i.e., the universal C*-algebra with respect to the representations of $\B$, and $\ca_\la(\B)$ for the reduced C*-algebra of $\B$, i.e., the C*-algebra generated by the left regular representation of $\B$ on $\ell^2(\B)$.
There is a canonical $*$-epimorphism $\ca(\B)\to\ca_\la(\B)$ that is isometric on every fibre $\B_g$.
Both $\ca(\B)$ and $\ca_\la(\B)$ admit a canonical coaction by $G$ whose $g$-spectral subspace is the copy of $\B_g$ for every $g\in G$, and the canonical $*$-epimorphism is equivariant with respect to these coactions.
Moreover, the coaction on $\ca_\la(\B)$ is normal.

\subsection{Reduced crossed products}\label{Ss:Cross}

We briefly discuss the properties of the reduced crossed product by a locally compact Hausdorff group that we are going to need later on.
For more details see \cite{EKQR06, Wil07}.

Let $\fH$ be a locally compact Hausdorff group with unit $e_{\fH}$.
We denote a left Haar measure on $\fH$ by $\mu$ and we write $\De$ for the modular function of $\fH$.
Let $\C$ be a C*-algebra and $\al \colon \fH \to \Aut(\C)$ be a point-norm continuous action.
We call the triple $(\C,\fH,\al)$ a \emph{C*-dynamical system}. A subset $\X\subseteq\C$ is called \emph{$\al$-invariant} if $\al_\fh(\X)\subseteq\X$ for all $\fh\in\fH$; since $\al_\fh^{-1}=\al_{\fh^{-1}}$, this is equivalent to $\al_\fh(\X)=\X$ for all $\fh\in\fH$.
The space $C_c(\fH,\C)$ of continuous $\C$-valued functions on $\fH$ with compact support becomes a $*$-algebra by setting
\[
(f \ast h)(\fg')
:=
\int_{\fH}f(\fg)\al_{\fg}(h(\fg^{-1}\fg')) \, d\mu(\fg)
\foral f,h\in C_c(\fH,\C),
\]
and
\[
f^*(\fg)
:=
\De(\fg)^{-1}\al_{\fg}(f(\fg^{-1})^*) \foral f\in C_c(\fH,\C).
\]
We write $\C\rtimes_{\al,\la}\fH$ for the \emph{reduced crossed product}, i.e., the completion of $C_c(\fH,\C)$ with respect to the norm induced by the regular representation; see \cite[Definition~7.7 and Lemma~7.8]{Wil07} for its construction.
We regard $C_c(\fH,\C)$ as a dense $*$-subalgebra of $\C\rtimes_{\al,\la}\fH$, and we recall that the norm of $\C\rtimes_{\al,\la}\fH$ is dominated by the $L^1$-norm on $C_c(\fH,\C)$.

Let $(\C,\fH, \al)$ and $(\C',\fH, \be)$ be two C*-dynamical systems and let $\Phi\colon\C\to\C'$ be an \emph{$\fH$-equivariant} $*$-homomorphism, i.e., a $*$-homomorphism such that
\[
\be_\fh \circ \Phi = \Phi \circ \al_\fh \foral \fh\in \fH.
\]
Then there is an induced $*$-homomorphism
\[
\Phi\rtimes\fH \colon \C\rtimes_{\al,\la}\fH \to \C'\rtimes_{\be,\la}\fH
\]
that extends the map $f\mapsto\Phi\circ f$ on $C_c(\fH,\C)$, and $\Phi\rtimes\fH$ is injective whenever $\Phi$ is injective; see for example \cite[Lemma~A.16]{EKQR06}.
Moreover, $\Phi\rtimes\fH$ is surjective whenever $\Phi$ is surjective.
Indeed, its range is closed and contains the functions $\fg\mapsto f(\fg)\Phi(c)$ for $f\in C_c(\fH)$ and $c\in\C$, which span a dense subspace of $\C'\rtimes_{\be,\la}\fH$ by \cite[Lemma~1.87]{Wil07}.
In particular, if $\C_0\subseteq\C$ is an $\al$-invariant C*-subalgebra, then the closure $[C_c(\fH,\C_0)]$ of $C_c(\fH,\C_0)$ inside $\C\rtimes_{\al,\la}\fH$ is canonically $*$-isomorphic to $\C_0\rtimes_{\al,\la}\fH$, and we will identify these two C*-algebras without further mention.

We briefly discuss reduced crossed products of Hilbert modules; see \cite[Section~3.1]{EKQR06} for more details.
Let $(\C,\fH,\al)$ be a C*-dynamical system.
An \emph{$\fH$-equivariant Hilbert $\C$-module} is a right Hilbert $\C$-module $E$ together with linear bijections $\ga_\fh\colon E\to E$, for $\fh\in\fH$, such that:
\begin{enumerate}
\item $\ga_{\fh\fg}=\ga_\fh\ga_\fg$ for all $\fh,\fg\in\fH$;
\item the map $\fh\mapsto\ga_\fh(\ze)$ is continuous for every $\ze\in E$; and
\item $\ga_\fh(\ze c)=\ga_\fh(\ze)\al_\fh(c)$ and $\sca{\ga_\fh(\ze),\ga_\fh(\eta)}=\al_\fh(\sca{\ze,\eta})$ for all $\ze,\eta\in E$, $c\in\C$ and $\fh\in\fH$.
\end{enumerate}
In this case, for every $\fh \in \fH$ the map
\[
(\ad\ga)_\fh \colon \L(E) \to \L(E);\ x\mapsto\ga_\fh\circ x\circ\ga_{\fh^{-1}}
\]
is a $*$-automorphism of the adjointable operators $\L(E)$, and thus we obtain an action $\ad\ga$ of $\fH$ on $\L(E)$.
We note that $\ad\ga$ need not be point-norm continuous on $\L(E)$. A subset $\X\subseteq E$ is called \emph{$\ga$-invariant} if $\ga_\fh(\X)=\X$ for all $\fh\in\fH$, and a subset $\X\subseteq\L(E)$ is called \emph{$(\ad\ga)$-invariant} if $(\ad\ga)_\fh(\X)=\X$ for all $\fh\in\fH$.
The \emph{reduced crossed product} $E\rtimes_{\ga,\la}\fH$ is the right Hilbert $(\C\rtimes_{\al,\la}\fH)$-module obtained as the completion of the pre-Hilbert module $C_c(\fH,E)$ with respect to the norm induced by the $(\C\rtimes_{\al,\la}\fH)$-valued inner product
\[
\sca{\ze,\eta}(\fg')
:=
\int_{\fH}\al_{\fg^{-1}}\big(\sca{\ze(\fg),\eta(\fg\fg')}\big)\,d\mu(\fg)\foral \ze,\eta\in C_c(\fH,E),
\]
and with its right module action defined by
\[
(\ze\cdot f)(\fg')
:=
\int_{\fH}\ze(\fg)\al_{\fg}(f(\fg^{-1}\fg'))\,d\mu(\fg) \foral \ze\in C_c(\fH,E), f\in C_c(\fH,\C).
\]
We note that for $\ze\in C_c(\fH,E)$ we have
\begin{equation}\label{Eq:Lone}
\|\ze\|_{E\rtimes_{\ga,\la}\fH}\leq\|\ze\|_1:=\int_{\fH}\|\ze(\fg)\|\,d\mu(\fg).
\end{equation}
Indeed, since the norm of $\C\rtimes_{\al,\la}\fH$ is dominated by the $L^1$-norm we get
\begin{align*}
\|\ze\|^2_{E\rtimes_{\ga,\la}\fH}
=\|\sca{\ze,\ze}\|_{\C\rtimes_{\al,\la}\fH}
& \leq \int_{\fH}\Big\|\int_{\fH}\al_{\fg^{-1}}\big(\sca{\ze(\fg),\ze(\fg\fg')}\big)\,d\mu(\fg)\Big\|\,d\mu(\fg') \\
& \leq \int_{\fH}\int_{\fH}\|\ze(\fg)\|\,\|\ze(\fg\fg')\|\,d\mu(\fg')\,d\mu(\fg)
= \|\ze\|_1^2.
\end{align*}

Now let $E$ be a right Hilbert $\C$-module.
Two submodules $E_1,E_2\subseteq E$ are called \emph{orthogonal} if $\sca{\xi,\eta}=0$ for all $\xi\in E_1$ and $\eta\in E_2$.
If $\{E_k\}_k$ is a family of mutually orthogonal closed submodules of $E$ whose algebraic sum is dense in $E$, then we say that $E$ is the \emph{internal direct sum} of the $E_k$, and we write $E=\sum^\oplus_k E_k$. Moreover, if $J$ is an ideal in $\C$ we write $EJ$ for $[E\cdot J]$.

We will need the following proposition, which is likely known to the experts. We include a proof for completeness.

\begin{proposition}\label{P:modsub}
Let $(\C,\fH,\al)$ be a C*-dynamical system and let $E$ be an $\fH$-equivariant Hilbert $\C$-module with respect to $\ga$.
Then the following hold:
\begin{enumerate}
\item If $E_0\subseteq E$ is a closed $\ga$-invariant submodule, then $E_0\rtimes_{\ga,\la}\fH$ is canonically the closure of $C_c(\fH,E_0)$ in $E\rtimes_{\ga,\la}\fH$.
\item If $E=\sum^\oplus_{k}E_k$ for closed $\ga$-invariant submodules $E_k$, then $E\rtimes_{\ga,\la}\fH=\sum^\oplus_{k}E_k\rtimes_{\ga,\la}\fH$, where every $E_k\rtimes_{\ga,\la}\fH$ is regarded as a submodule of $E\rtimes_{\ga,\la}\fH$ by item (i).
\item If $J\subseteq\C$ is an $\al$-invariant ideal, then
\[
EJ\rtimes_{\ga,\la}\fH\subseteq[(E\rtimes_{\ga,\la}\fH)\cdot(J\rtimes_{\al,\la}\fH)].
\]
\end{enumerate}
\end{proposition}

\begin{proof}
For item (i), let $E_0\subseteq E$ be a closed $\ga$-invariant submodule, and thus $E_0$ is an $\fH$-equivariant Hilbert $\C$-module with respect to the restriction of $\ga$.
If $\ze,\eta\in C_c(\fH,E_0)$ and $f\in C_c(\fH,\C)$, then $\ze\cdot f$ and $\sca{\ze,\eta}$ are given by the same formulas when computed in $E_0\rtimes_{\ga,\la}\fH$ and in $E\rtimes_{\ga,\la}\fH$.
In particular, the norms of $E_0\rtimes_{\ga,\la}\fH$ and of $E\rtimes_{\ga,\la}\fH$ coincide on $C_c(\fH,E_0)$ and thus the inclusion $C_c(\fH,E_0)\hookrightarrow C_c(\fH,E)$ extends to an isometric linear map $E_0\rtimes_{\ga,\la}\fH\to E\rtimes_{\ga,\la}\fH$.

For item (ii), let $E=\sum^\oplus_k E_k$ for closed $\ga$-invariant submodules $E_k$.
By item (i) we regard every $E_k\rtimes_{\ga,\la}\fH$ as the closure of $C_c(\fH,E_k)$ in $E\rtimes_{\ga,\la}\fH$. It is easy to check that the $C_c(\fH,E_k)$ are mutually orthogonal, as the $E_k$ are, and hence it remains to show that their algebraic sum is dense in $E\rtimes_{\ga,\la}\fH$. For this, it suffices to approximate elements in $C_c(\fH,E)$ as this space is dense in $E\rtimes_{\ga,\la}\fH$.
Let $E_{\textup{alg}}$ be the algebraic sum of the $E_k$.
Every $\xi\in E_{\textup{alg}}$ decomposes uniquely as $\xi=\sum_k\xi_k$ with $\xi_k\in E_k$ and only finitely many $\xi_k$ non-zero.
For a finite set $F$ of indices we define the linear map
\[
p_F\colon E_{\textup{alg}}\to E_{\textup{alg}};\ \sum_k\xi_k\mapsto\sum_{k\in F}\xi_k.
\]
By orthogonality we have
\[
\Big\|\sum_{k\in F}\xi_k\Big\|^2
=\Big\|\sum_{k\in F}\sca{\xi_k,\xi_k}\Big\|
\leq\Big\|\sum_k\sca{\xi_k,\xi_k}\Big\|
=\Big\|\sum_k\xi_k\Big\|^2,
\]
as $0\leq\sum_{k\in F}\sca{\xi_k,\xi_k}\leq\sum_k\sca{\xi_k,\xi_k}$, and similarly $\|\sum_{k\notin F}\xi_k\|\leq\|\sum_k\xi_k\|$.
Hence $p_F$ extends to a contractive linear map $p_F\colon E\to E$ such that
\[
\|\xi-p_F\xi\|\leq\|\xi\| \foral \xi\in E.
\]
Moreover, we have $p_F=\sum_{k\in F}p_{\{k\}}$ on $E$ by continuity, and $p_{\{k\}}(E)\subseteq E_k$ as $E_k$ is closed.

Now let $\ze$ be in $C_c(\fH,E)$ and $\eps>0$.
As the set $\ze(\fH)$ is compact, there are $\xi_1,\dots,\xi_n\in\ze(\fH)$ such that every element of $\ze(\fH)$ is at distance less than $\eps$ from some $\xi_i$.
As $E_{\textup{alg}}$ is dense in $E$, there are $\eta_1,\dots,\eta_n\in E_{\textup{alg}}$ such that $\|\xi_i-\eta_i\|<\eps$ for $i=1,\dots,n$.
Pick a finite set $F_0$ such that $p_F\eta_i=\eta_i$ for all $i=1,\dots,n$ and every finite set $F\supseteq F_0$.
Let $F\supseteq F_0$ be finite and $\fg\in\fH$, and choose $i$ such that $\|\ze(\fg)-\xi_i\|<\eps$.
Then
\[
\|\ze(\fg)-p_F(\ze(\fg))\|
\leq\|(\ze(\fg)-\xi_i)-p_F(\ze(\fg)-\xi_i)\|+\|(\xi_i-\eta_i)-p_F(\xi_i-\eta_i)\|+\|\eta_i-p_F\eta_i\|
<2\eps.
\]
The function $p_F\circ\ze$ is continuous and compactly supported, and so \eqref{Eq:Lone} gives
\[
\|\ze-p_F\circ\ze\|_{E\rtimes_{\ga,\la}\fH}
\leq\int_{\fH}\|\ze(\fg)-p_F(\ze(\fg))\|\,d\mu(\fg)
\leq2\eps\,\mu(\supp\ze).
\]
As $p_F\circ\ze=\sum_{k\in F}p_{\{k\}}\circ\ze$ with $p_{\{k\}}\circ\ze\in C_c(\fH,E_k)$, and $\eps>0$ was arbitrary, we conclude that $\ze$ lies in the closure of the algebraic sum of the $E_k\rtimes_{\ga,\la}\fH$.

For item (iii), let $J\subseteq\C$ be an $\al$-invariant ideal.
As $\al_\fg(J)=J$ for all $\fg\in\fH$, we have
\[
C_c(\fH,J)\ast C_c(\fH,\C)+C_c(\fH,\C)\ast C_c(\fH,J)\subseteq C_c(\fH,J),
\]
and therefore $J\rtimes_{\al,\la}\fH=[C_c(\fH,J)]$ is an ideal of $\C\rtimes_{\al,\la}\fH$.
Moreover, since $J$ is an $\al$-invariant ideal we get that $[E\cdot J]$ is a closed submodule of $E$.
By item (i), we have that $[E\cdot J]\rtimes_{\ga,\la}\fH$ is the closure of $C_c(\fH,[E\cdot J])$ in $E\rtimes_{\ga,\la}\fH$, and thus it suffices to show that
\[
C_c(\fH,[E\cdot J])\subseteq[(E\rtimes_{\ga,\la}\fH)\cdot(J\rtimes_{\al,\la}\fH)].
\]
For $\xi,\eta\in E$ and $a\in J$ we have $\sca{\xi a,\eta}=a^*\sca{\xi,\eta}\in J$, and hence $\sca{[E\cdot J],E}\subseteq J$.
Now let $\ze$ be in $C_c(\fH,[E\cdot J])$.
In particular, as $J$ is $\al$-invariant, the integrand in
\[
\sca{\ze,\ze}(\fg')=\int_{\fH}\al_{\fg^{-1}}\big(\sca{\ze(\fg),\ze(\fg\fg')}\big)\,d\mu(\fg)
\]
takes values in $J$, and so $\sca{\ze,\ze}\in C_c(\fH,J)\subseteq J\rtimes_{\al,\la}\fH$.
Let $(v_i)_i$ be a positive contractive approximate unit of $J\rtimes_{\al,\la}\fH$.
Computing in the unitisation of $\C\rtimes_{\al,\la}\fH$ and using that $0\leq1-v_i\leq1$, we obtain
\[
\|\ze-\ze\cdot v_i\|^2
=\|\sca{\ze-\ze\cdot v_i,\ze-\ze\cdot v_i}\|
=\|(1-v_i)\sca{\ze,\ze}(1-v_i)\|
\leq\|\sca{\ze,\ze}-v_i\sca{\ze,\ze}\|
\xrightarrow{i}0.
\]
Since $\ze\cdot v_i\in(E\rtimes_{\ga,\la}\fH)\cdot(J\rtimes_{\al,\la}\fH)$ for every $i$, we conclude that $\ze\in[(E\rtimes_{\ga,\la}\fH)\cdot(J\rtimes_{\al,\la}\fH)]$, as required.
\end{proof}

The following proposition is likely known to the experts. We include a proof for completeness.

\begin{proposition}\label{P:modrep}
Let $(\C,\fH,\al)$ be a C*-dynamical system, let $E$ be an $\fH$-equivariant Hilbert $\C$-module with respect to $\ga$, and let $\D\subseteq\L(E)$ be an $(\ad\ga)$-invariant C*-subalgebra such that $\ad\ga$ is point-norm continuous on $\D$.
Then the formula
\[
(f\cdot\ze)(\fg')
:=
\int_{\fH}f(\fg)\ga_{\fg}(\ze(\fg^{-1}\fg'))\,d\mu(\fg)
\foral f\in C_c(\fH,\D) \text{ and } \ze\in C_c(\fH,E)
\]
extends to a faithful $*$-representation $\ka \colon \D\rtimes_{\ad\ga,\la}\fH\to\L(E\rtimes_{\ga,\la}\fH)$.
\end{proposition}

\begin{proof}
First note that since $\ad\ga$ is point-norm continuous on $\D$, the triple $(\D,\fH,\ad\ga)$ is a C*-dynamical system, and so $\D\rtimes_{\ad\ga,\la}\fH$ is well defined.
For a norm-continuous compactly supported function $f\colon\fH\to\L(E)$ and $\ze\in C_c(\fH,E)$, the function $(\fg,\fg')\mapsto f(\fg)\ga_\fg(\ze(\fg^{-1}\fg'))$ is a continuous compactly supported $E$-valued function.
Hence the formula
\[
(f\cdot\ze)(\fg'):=\int_{\fH}f(\fg)\ga_\fg(\ze(\fg^{-1}\fg'))\,d\mu(\fg)
\]
defines an element $f\cdot\ze\in C_c(\fH,E)$.

For $\xi,\eta\in E$ and $\fh\in\fH$ we have
\[
(\ad\ga)_\fh(\theta_{\xi,\eta})=\theta_{\ga_\fh(\xi),\ga_\fh(\eta)},
\]
and hence $\K(E)$ is $(\ad\ga)$-invariant.
Moreover, we have
\[
\|\theta_{\ga_\fh(\xi),\ga_\fh(\eta)}-\theta_{\xi,\eta}\|\leq\|\ga_\fh(\xi)-\xi\|\,\|\ga_\fh(\eta)\|+\|\xi\|\,\|\ga_\fh(\eta)-\eta\|\xrightarrow{\fh\to e_\fH}0,
\]
and thus it follows that $\ad\ga$ is point-norm continuous on $\K(E)$.
Since the reduced crossed product functor is a correspondence crossed product functor, by \cite[Lemma~4.22]{BEW18} there is a canonical $*$-isomorphism
\[
\Phi\colon\K(E)\rtimes_{\ad\ga,\la}\fH\to\K(E\rtimes_{\ga,\la}\fH)
;\
\Phi(h)\ze=h\cdot\ze \foral h\in C_c(\fH,\K(E)) \text{ and }\ze\in C_c(\fH,E).
\]
Let 
\[
\ol{\Phi}\colon\M(\K(E)\rtimes_{\ad\ga,\la}\fH)\to\M(\K(E\rtimes_{\ga,\la}\fH))=\L(E\rtimes_{\ga,\la}\fH)
\]
be the unique extension of $\Phi$ to a $*$-isomorphism of the multiplier algebras, so that 
\[
\ol{\Phi}(m)\Phi(x)=\Phi(mx) \foral m\in\M(\K(E)\rtimes_{\ad\ga,\la}\fH) \text{ and } x\in\K(E)\rtimes_{\ad\ga,\la}\fH.
\]

For $f\in C_c(\fH,\D)$ and $h\in C_c(\fH,\K(E))$ we write
\[
(f\ast h)(\fg'):=\int_{\fH}f(\fg)(\ad\ga)_\fg(h(\fg^{-1}\fg'))\,d\mu(\fg),
\]
and we note that $f\ast h\in C_c(\fH,\K(E))$, as $\K(E)$ is an $(\ad\ga)$-invariant ideal of $\L(E)$.
Now let $j\colon\D\hookrightarrow\L(E)\simeq\M(\K(E))$ be the inclusion.
The $*$-homomorphism $j$ is injective, and it is $\fH$-equivariant (with respect to the definition given in \cite[Comments prior to Example~A.12]{EKQR06}), as the action $\ad\ga$ on $\L(E)$ is the canonical extension of the action $\ad\ga$ on $\K(E)$ to $\M(\K(E))$.
By \cite[Lemma~A.16]{EKQR06} there is an injective $*$-homomorphism
\[
j\rtimes\fH\colon\D\rtimes_{\ad\ga,\la}\fH\to\M(\K(E)\rtimes_{\ad\ga,\la}\fH)
\]
that satisfies
\[
(j\rtimes\fH)(f)h=f\ast h \foral f\in C_c(\fH,\D) \text{ and } h\in C_c(\fH,\K(E)).
\]
We define the faithful $*$-representation
\[
\ka:=\ol{\Phi}\circ(j\rtimes\fH)\colon\D\rtimes_{\ad\ga,\la}\fH\to\L(E\rtimes_{\ga,\la}\fH).
\]
By construction, for all $f\in C_c(\fH,\D)$ and $h\in C_c(\fH,\K(E))$ we have
\begin{equation}\label{Eq:kacomp}
\ka(f)\Phi(h)=\ol{\Phi}\big((j\rtimes\fH)(f)\big)\Phi(h)=\Phi\big((j\rtimes\fH)(f)h\big)=\Phi(f\ast h).
\end{equation}

It remains to show that $\ka(f)\ze=f\cdot\ze$ for all $f\in C_c(\fH,\D)$ and $\ze\in C_c(\fH,E)$.
For this we need two identities.
Fix $f\in C_c(\fH,\D)$, $h\in C_c(\fH,\K(E))$ and $\ze,\eta\in C_c(\fH,E)$.
First, using that $\ga_\fg\circ h(\fs)=(\ad\ga)_\fg(h(\fs))\circ\ga_\fg$, and then the left-invariance of the Haar measure and the substitution $\fs\mapsto\fg^{-1}\fs$, we have
\begin{align*}
(f\cdot(h\cdot\ze))(\fg')
& = \int_{\fH}f(\fg)\ga_\fg\Big(\int_{\fH}h(\fs)\ga_\fs(\ze(\fs^{-1}\fg^{-1}\fg'))\,d\mu(\fs)\Big)\,d\mu(\fg)\\
& = \int_{\fH}\int_{\fH}f(\fg)(\ad\ga)_\fg(h(\fs))\ga_{\fg\fs}(\ze((\fg\fs)^{-1}\fg'))\,d\mu(\fs)\,d\mu(\fg)\\
& = \int_{\fH}\int_{\fH}f(\fg)(\ad\ga)_\fg(h(\fg^{-1}\fs))\ga_{\fs}(\ze(\fs^{-1}\fg'))\,d\mu(\fg)\,d\mu(\fs)\\
& = \int_{\fH}(f\ast h)(\fs)\ga_\fs(\ze(\fs^{-1}\fg'))\,d\mu(\fs)
= ((f\ast h)\cdot\ze)(\fg'),
\end{align*}
and therefore
\begin{equation}\label{Eq:assoc}
f\cdot(h\cdot\ze)=(f\ast h)\cdot\ze.
\end{equation}
Secondly, using that $\sca{\ga_\fg(x),y}=\al_\fg(\sca{x,\ga_{\fg^{-1}}(y)})$ for $x,y\in E$, and then the substitution $\fs\mapsto\fg\fs$, we have
\begin{align*}
\sca{f\cdot\ze,\eta}(\fg')
& = \int_{\fH}\int_{\fH}\al_{\fs^{-1}}\big(\sca{f(\fg)\ga_\fg(\ze(\fg^{-1}\fs)),\eta(\fs\fg')}\big)\,d\mu(\fg)\,d\mu(\fs)\\
& = \int_{\fH}\int_{\fH}\al_{\fs^{-1}\fg}\big(\sca{\ze(\fg^{-1}\fs),\ga_{\fg^{-1}}(f(\fg)^*\eta(\fs\fg'))}\big)\,d\mu(\fs)\,d\mu(\fg)\\
& = \int_{\fH}\int_{\fH}\al_{\fs^{-1}}\big(\sca{\ze(\fs),\ga_{\fg^{-1}}(f(\fg)^*\eta(\fg\fs\fg'))}\big)\,d\mu(\fs)\,d\mu(\fg)\\
& = \int_{\fH}\al_{\fs^{-1}}\Big(\Big\langle\ze(\fs),\int_{\fH}\ga_{\fg^{-1}}\big(f(\fg)^*\eta(\fg\fs\fg')\big)\,d\mu(\fg)\Big\rangle\Big)\,d\mu(\fs).
\end{align*}
On the other hand, by the identity $\int_{\fH}\phi(\fg)\,d\mu(\fg)=\int_{\fH}\De(\fg)^{-1}\phi(\fg^{-1})\,d\mu(\fg)$ we have
\begin{align*}
(f^*\cdot\eta)(\fs\fg')
& = \int_{\fH}\De(\fg)^{-1}\ga_\fg\big(f(\fg^{-1})^*\eta(\fg^{-1}\fs\fg')\big)\,d\mu(\fg)
= \int_{\fH}\ga_{\fg^{-1}}\big(f(\fg)^*\eta(\fg\fs\fg')\big)\,d\mu(\fg),
\end{align*}
and therefore
\begin{equation}\label{Eq:adj}
\sca{f\cdot\ze,\eta}(\fg')=\int_{\fH}\al_{\fs^{-1}}\big(\sca{\ze(\fs),(f^*\cdot\eta)(\fs\fg')}\big)\,d\mu(\fs)=\sca{\ze,f^*\cdot\eta}(\fg').
\end{equation}

Now let $\S$ be the linear span of the elements $h\cdot\eta$ for $h\in C_c(\fH,\K(E))$ and $\eta\in C_c(\fH,E)$.
We claim that $\S$ is dense in $E\rtimes_{\ga,\la}\fH$.
Indeed, since $h\cdot\eta=\Phi(h)\eta$, and $\Phi(C_c(\fH,\K(E)))$ and $C_c(\fH,E)$ are dense in $\K(E\rtimes_{\ga,\la}\fH)$ and in $E\rtimes_{\ga,\la}\fH$ respectively, the closure of $\S$ contains $\K(E\rtimes_{\ga,\la}\fH)(E\rtimes_{\ga,\la}\fH)$, which proves the claim.
Finally, consider $f\in C_c(\fH,\D)$ and $\ze\in C_c(\fH,E)$.
For $h\in C_c(\fH,\K(E))$ and $\eta\in C_c(\fH,E)$, the equations \eqref{Eq:kacomp}, \eqref{Eq:assoc} and \eqref{Eq:adj} give
\[
\sca{\ka(f)\ze,h\cdot\eta}
=\sca{\ze,\ka(f^*)\Phi(h)\eta}
=\sca{\ze,(f^*\ast h)\cdot\eta}
=\sca{\ze,f^*\cdot(h\cdot\eta)}
=\sca{f\cdot\ze,h\cdot\eta}.
\]
By density we get $\ka(f)\ze= f \cdot \ze$, as required, and the proof is complete.
\end{proof}

\subsection{Operator algebras}

For a comprehensive treatment and full details regarding the general theory of operator algebras, we refer the reader to \cite{BL04, Pau02}.

By an operator algebra $\A$, we mean a norm-closed subalgebra of some $\B(H)$.
Any operator algebra possesses a C*-cover, i.e., a pair $(\C,j)$, where $\C$ is a C*-algebra and $j \colon \A \to \C$ is a completely isometric homomorphism satisfying $\C = \ca(j(\A))$.
The \emph{C*-envelope} $(\cenv(\A),\io)$ is the co-universal C*-cover of $\A$.
This means that, given any other C*-cover $(\C,j)$, there exists a unique $*$-epimorphism $\Phi \colon \C \to \cenv(\A)$ such that $\Phi \circ j = \io$.
Hamana established the existence of the C*-envelope in \cite{Ham79}. For the non-unital case the reader is referred to \cite{BL04}.
The kernel of $\Phi$ is called the \emph{Shilov boundary ideal} of $\A$ in $\C$, and it is the largest ideal $J$ of $\C$ such that the quotient map $\C\to\C/J$ is completely isometric on $j(\A)$.

Following the work of Katsoulis and Ramsey \cite{KR19}, we recall the definition of reduced crossed products of operator algebras.
Suppose that $\fH$ is a locally compact Hausdorff group acting point-norm continuously on an operator algebra $\A$ via an action $\al$ by completely isometric automorphisms.
This action admits an extension to an action $\dot{\al}$ on the C*-envelope of $\A$ by co-universality; see for example \cite[Lemma~3.4]{KR19}. This allows us to construct the reduced crossed product $\cenv(\A) \rtimes_{\dot\al, \la} \fH$.
By identifying $\A$ with its copy inside $\cenv(\A)$, the \emph{reduced crossed product} $\A \rtimes_{\al, \la} \fH$ is defined to be the norm-closed subalgebra of $\cenv(\A) \rtimes_{\dot{\al}, \la} \fH$ generated by the compactly supported continuous functions $C_c(\fH,\A)$; see \cite[Definition~3.17]{KR19}.

\begin{remark}\label{R:cai}
In \cite{KR19} all operator algebras are assumed to admit a contractive approximate unit, whereas the tensor algebra of a product system (see Subsections~\ref{Ss:Prodsyst} and \ref{Ss:sc}) that is not non-degenerate need not admit one.
We therefore note that \cite[Corollary~3.16]{KR19}, which we will use in the sequel, holds for arbitrary operator algebras.
Let $\fH$ act on an operator algebra $\A$ via $\al$ as above.
Let $(\C,j)$ be a C*-cover of $\A$ which is \emph{$\al$-admissible}, i.e., the action $\fh\mapsto j\circ\al_\fh\circ j^{-1}$ of $\fH$ on $j(\A)$ extends to an action of $\fH$ on $\C$ by $*$-automorphisms, which is automatically point-norm continuous and is denoted again by $\al$.
Then the $*$-epimorphism $\Phi\colon\C\to\cenv(\A)$ with $\Phi\circ j=\io$ is $\fH$-equivariant, and \cite[Lemma~3.12]{KR19} yields that, if $\A$ is unital, then the induced $*$-epimorphism
\[
\Phi\rtimes\fH\colon\C\rtimes_{\al,\la}\fH\to\cenv(\A)\rtimes_{\dot\al,\la}\fH
\]
restricts to a completely isometric isomorphism from the norm-closed subalgebra of $\C\rtimes_{\al,\la}\fH$ generated by $C_c(\fH,j(\A))$ onto $\A\rtimes_{\al,\la}\fH$.
For an arbitrary non-unital operator algebra $\A$, let $\C^1$ be the unitisation of $\C$ if $\C$ is not unital and $\C^1:=\C$ otherwise, and set $\A^1:=j(\A)+\bC1\subseteq\C^1$.
By Meyer's Unitisation Theorem \cite[Corollary~2.1.15]{BL04} we get that $\A^1$ is the unitisation of $\A$, and $\C^1$ is an admissible C*-cover of $\A^1$ with respect to the unital extension of $\al$.
By \cite[Proposition~4.3.5]{BL04} we have that $\cenv(\A)$ is the C*-algebra generated by the copy of $\A$ inside $\cenv(\A^1)$. The map $\Phi$ is the restriction to $\C$ of the induced $*$-epimorphism $\C^1\to\cenv(\A^1)$.
As $\C\rtimes_{\al,\la}\fH\subseteq\C^1\rtimes_{\al,\la}\fH$, the assertion for $\A$ follows by applying \cite[Lemma~3.12]{KR19} to $\A^1$.
\end{remark}

\subsection{Product systems}\label{Ss:Prodsyst}

In his seminal paper, Fowler \cite{Fow02} formalised product systems, motivated by the works of Nica \cite{Nic92}, Pimsner \cite{Pim97}, as well as his joint work with Raeburn \cite{FR98}.
We will work in the framework of concrete product systems, as considered in \cite{KP25}.
For a more detailed description of concrete product systems, alongside a justification of how they generalise the product systems originally formulated in \cite{Fow02}, see \cite[Section~2]{KP25}.

Let $P$ be a unital subsemigroup of a discrete group $G$.
We say that a family $X = \{X_p\}_{p \in P}$ of norm-closed operator spaces in a common $\B(H)$ is a \emph{(concrete) product system} if the following are satisfied:
\begin{enumerate}
\item $A := X_{e_P}$ is a C*-algebra (called the \emph{coefficient algebra of $X$});
\item $X_p \cdot X_q \subseteq X_{pq}$ for all $p, q \in P$;
\item $X_p^* \cdot X_{pq} \subseteq X_q$ for all $p, q \in P$.
\end{enumerate}
Uniqueness of $q \in P$ in item (iii) follows since $P$ is left-cancellative.
Each $X_p$ is then a C*-correspondence over $A$ with inner product $\sca{\xi_p, \eta_p} := \xi_p^* \cdot \eta_p$ and left action $\vphi_{p}\colon A\to\L(X_p)$ given by multiplication.
If, in addition, $[A\cdot X_p]=X_p$ for every $p\in P$, then $X$ is called \emph{non-degenerate}.
We stress that non-degeneracy of $X$ is not assumed in this work.

For $\xi_p \in X_p$, the multiplication operator
\[
M_{\xi_p}^{q, pq} \colon X_q \to X_{pq};\ \eta_q \mapsto \xi_p \cdot \eta_q,
\]
lies in $\L(X_q, X_{pq})$, with adjoint given by $\eta_{pq} \mapsto \xi_p^* \cdot \eta_{pq}$.
The Fock representation of a product system was first considered in \cite{Fow02}.
Consider the \emph{Fock space} $\F X := \sum^\oplus_{r \in P} X_r$, as a right Hilbert $A$-module.
Here $\sum^\oplus_{r\in P}X_r$ denotes the external direct sum of the Hilbert $A$-modules $X_r$, i.e., the space of families $(\xi_r)_{r\in P}$ with $\xi_r\in X_r$ such that $\sum_{r\in P}\sca{\xi_r,\xi_r}$ converges in norm in $A$, equipped with the coordinatewise operations and the inner product $\sca{(\xi_r)_{r\in P},(\eta_r)_{r\in P}}:=\sum_{r\in P}\sca{\xi_r,\eta_r}$; see \cite{Lan95}.
We identify every $X_r$ with the submodule of $\F X$ that consists of the families supported at $r$.
Then the $X_r$ are mutually orthogonal closed submodules of $\F X$ whose algebraic sum is dense, and thus $\F X$ is also the internal direct sum of the $X_r$.
For every $\xi_p\in X_p$, we define the \emph{left creation operator}
\[
\la_p(\xi_p) := \sum^\oplus_{r \in P} M_{\xi_p}^{r, pr} \colon \F X \to \F X,
\]
so that for $\eta_r \in X_r$ we have
\[
\la_p(\xi_p) \eta_r = \xi_p \cdot \eta_r
\qand
\la_p(\xi_p)^* \eta_r =
\begin{cases}
\xi_p^* \cdot \eta_r & \text{if } r \in pP, \\
0 & \text{if } r \notin pP.
\end{cases}
\]
We refer to $\la = \{\la_p\}_{p \in P}$ as the \emph{Fock representation of $X$}, and we also define the \emph{Fock C*-algebra of $X$} to be the C*-algebra $\T_\la(X) := \ca(\la_p(X_p) \mid p \in P)$.

A \emph{representation} $t = \{t_p\}_{p \in P}$ of $X$ is a family of linear maps $t_p \colon X_p \to \C$, where $\C$ is a C*-algebra, such that:
\begin{enumerate}
\item $t_{e_P}$ is a $*$-representation of $A$;
\item $t_p(\xi_p) t_q(\xi_q) = t_{pq}(\xi_p \cdot \xi_q)$ for all $\xi_p \in X_p$ and $\xi_q \in X_q$;
\item $t_p(\xi_p)^* t_{pq}(\xi_{pq}) = t_q(\xi_p^* \cdot \xi_{pq})$ for all $\xi_p \in X_p$ and $\xi_{pq} \in X_{pq}$.
\end{enumerate}
A representation $t$ is called \emph{injective} if $t_{e_P}$ is injective on $A$.
It is clear that the Fock representation is an injective representation of $X$.
The \emph{Toeplitz algebra $\T(X)$ of $X$} is the universal C*-algebra generated by the \emph{universal representation $\wh{t}$ of $X$} with respect to the representations of $X$.
If $t$ is a representation of $X$, then there is an induced $*$-representation
\[
t_\ast \colon \T(X) \to \ca(t);\ \wh{t}_p(\xi_p)\mapsto t_p(\xi_p),
\]
where we write $\ca(t) := t_\ast(\T(X)) = \ca(t_p(X_p) \mid p \in P)$.

A representation $t$ of $X$ is called \emph{$G$-equivariant} if there exists a $*$-homomorphism
\[
\de \colon \ca(t) \to \ca(t) \otimes \ca(G);\ t_p(\xi_p) \mapsto t_p(\xi_p) \otimes u_p,
\]
where $\{u_g\}_{g\in G}$ are the unitaries that generate the group C*-algebra $\ca(G)$.
It follows that $\de$ is injective and satisfies the coaction identity, and hence $\ca(t)$ admits a coaction by $G$.
In this case, for $g\in G$ the \emph{$g$-spectral subspace}
\[
[\ca(t)]_g:=\{c\in\ca(t)\mid\de(c)=c\otimes u_g\}
\]
is the closed linear span of the spaces
\[
\left( t_{p_1}(X_{p_1})^* \right)^{\eps} t_{q_1}(X_{q_1}) \cdots
t_{p_n}(X_{p_n})^* t_{q_n}(X_{q_n})^{\eps'}
\text{ such that }
p_1^{-\eps} q_1 \cdots p_n^{-1} q_n^{\eps'} = g,
\]
for $\eps, \eps' \in \{0,1\}$ and $n \in \bN$, where the factors with exponent $0$ are omitted.
A proof of this realisation of the $g$-spectral subspace can be found in \cite[Lemma~2.2]{Seh19} for non-degenerate product systems, and similar arguments give the conclusion in the general case.
We call the $e_P$-spectral subspace $[\ca(t)]_{e_P}$ the \emph{fixed point algebra} of $\ca(t)$.
For a $G$-equivariant representation $t$ of $X$, the collection $\{[\ca(t)]_g\}_{g\in G}$ forms a Fell bundle over $G$.
In \cite[Lemma~2.2]{Seh19} it is shown that the universal representation $\wh{t}$ is $G$-equivariant.

We say that a representation \emph{$t$ admits a normal coaction} if it is $G$-equivariant with respect to $\de$ and the map $(\id \otimes \la_G) \circ \de$ is injective, where $\la_G \colon \ca(G) \to \ca_\la(G)$ denotes the left regular representation of $G$.
In this case the conditional expectation of $\ca(t)$ onto $[\ca(t)]_{e_P}$ defined as
\[
\E:=(\id\otimes \om_\la) \circ (\id\otimes \la_G) \circ \de,
\]
where $\om_\la$ is the canonical faithful tracial state on $\ca_\la(G)$, is faithful. In particular, by \cite[Proposition~3.7]{Exe97} we obtain that for any $t$ that admits a normal coaction, the C*-algebra $\ca(t)$ is canonically $*$-isomorphic to the reduced C*-algebra of the Fell bundle $\{[\ca(t)]_g\}_{g\in G}$.
In \cite[Proposition~4.1]{DKKLL22} it is shown that $\la$ admits a normal coaction by $G$.

An ideal $\I$ of $\T(X)$ is called \emph{induced} if $\I=\sca{\I\cap[\T(X)]_{e_P}}$.
In this case, the quotient map $\T(X)\to\T(X)/\I$ is $G$-equivariant, and so $\T(X)/\I$ inherits a coaction by $G$; see for example \cite[Section~2.1]{KP25}.
Since the Fock representation admits a normal coaction, we obtain a Fell bundle
\[
\F\C X := \Big\{ [\T_\la(X)]_g \Big\}_{g \in G},
\]
called the \emph{Fock covariant bundle} of $X$ \cite{DKKLL22}, and we have $\T_\la(X)\simeq\ca_\la(\F\C X)$.
Following \cite[Section~3.1]{KP25}, we write $\T_{\cov}^{\fock}(X)$ for the quotient of $\T(X)$ by the induced ideal $\sca{\ker\la_\ast\cap[\T(X)]_{e_P}}$, which is canonically $*$-isomorphic to the full C*-algebra of $\F\C X$.
A representation $t$ of $X$ is called \emph{Fock covariant} if $t_\ast$ factors through $\T_{\cov}^{\fock}(X)$.
In particular, the Fock representation is Fock covariant.

\subsection{Strong covariant representations}\label{Ss:sc}

We review the key elements of the strong covariant C*-algebra of a product system introduced by Sehnem \cite{Seh19, Seh22}, and its Fell bundle considered by Dor-On, Kakariadis, Katsoulis, Laca and Li \cite{DKKLL22}.
The strong covariant algebra of Sehnem \cite{Seh19}, as well as the reduced variant considered in \cite{DKKLL22}, encodes the previously known constructions for a co-universal or a Cuntz-type C*-algebra; see \cite{DKKLL22, Seh19, Seh22} and the references therein for these connections.

Let $P$ be a unital subsemigroup of a discrete group $G$ and let $X$ be a product system over $P$.
Set $A:=X_{e_P}$, and for each $p\in P$ let $\vphi_p\colon A\to \L(X_p)$ be the $*$-homomorphism that implements the left action of $A$ on the C*-correspondence $X_p$.
For a finite set $F \subseteq G$ let
\[
K_F := \bigcap_{g \in F} gP.
\]
For $r \in P$ and $g \in G$ define the ideal of $A$ given by
\[
I_{r^{-1} K_{\{r,g\}}} :=
\begin{cases}
\bigcap\limits_{s \in K_{\{r,g\}}} \ker \vphi_{r^{-1} s} & \text{if } K_{\{r,g\}} \neq \mt \text{ and } r \notin K_{\{r,g\}},\\
\hspace{.5cm} A & \text{otherwise},
\end{cases}
\]
and for a finite set $F\subseteq G$ set
\[
I_{r^{-1} (r \vee F)} := \bigcap_{g \in F} I_{r^{-1} K_{\{r,g\}}}.
\]
For $g\in G$ we write $gF:=\{gg' \mid g'\in F\}$.
We have that $I_{r^{-1}(r \vee F)} = I_{(pr)^{-1}(pr \vee pF)}$ for all $r,p \in P$, and $I_{r^{-1}(r \vee F)} = I_{(s^{-1}r)^{-1}(s^{-1}r \vee s^{-1}F)}$ for all $s\in P$ and $r \in sP$.
Moreover
\[
I_{r^{-1}(r \vee F_1)} \subseteq I_{r^{-1}(r \vee F_2)} \; \text{ when } \; F_1 \supseteq F_2.
\]
We declare $K_\mt= \mt$ and $I_{r^{-1} (r \vee \mt)} = A$.
For a finite set $F \subseteq G$, consider the C*-correspondences
\[
X_F := \sum^\oplus_{r \in P} X_r I_{r^{-1} (r \vee F)}
\qand
X_F^+ := \sum^\oplus_{g \in G} X_{gF}.
\]
We declare $X_\mt = X_\mt^+ = \F X$.
Each $X_F^+$ is reducing for the coaction
\[
\de \colon \T_\la(X) \to \T_\la(X) \otimes \ca(G);\ \la_p(\xi_p) \mapsto \la_p(\xi_p) \otimes u_p,
\]
giving rise to a $*$-representation
\[
\Phi_F \colon \T_\la(X) \to \L( X_F^+ ) ;\ \la_p(\xi_p) \mapsto (\la_p(\xi_p) \otimes u_p)|_{X_F^+},
\]
where $\ca(G)$ acts on $\ell^2(G)$ by the left regular representation.
Here we make the identification
\[
X_{gF} \to X_{gF} \otimes \de_g ;\ \xi_r a_r \mapsto \xi_r a_r \otimes \de_g, \text{ for } \xi_r \in X_r, a_r \in I_{r^{-1}(r \vee gF)}.
\]
Moreover $X_F \subseteq X_F^+$ is reducing for $\Phi_F([\T_\la(X)]_{e_P})$ and so we obtain the representation
\[
\bigoplus\limits_{\textup{fin } F \subseteq G} \Phi_F(\cdot)|_{X_F} \colon [\T_\la(X)]_{e_P} \to \prod\limits_{\text{fin } F \subseteq G} \L(X_F).
\]
We fix the ideal $\I_{\scv, e}$ in $[\T(X)]_{e_P}$ by using the corona universe, namely
\[
b \in \I_{\scv, e} \qiff \bigoplus\limits_{\textup{fin } F \subseteq G} \Phi_F(\la_\ast(b))|_{X_F} \in c_0 \left( \L(X_F) \mid \textup{fin } F \subseteq G \right),
\]
where the finite subsets of $G$ are directed by inclusion.
We write $A \times_X P$ for the $G$-equivariant quotient of $\T(X)$ by the induced ideal $\I_{\scv} := \sca{ \I_{\scv, e} }$, and we write
\[
q_{{\rm sc}} \colon \T(X) \to A \times_X P
\]
for the quotient map.
It is shown in \cite{Seh19} that this construction does not depend on $G$.
Moreover, we define the \emph{strong covariant bundle} of $X$ to be the Fell bundle
\[
\S\C X := \Big\{ [A \times_X P]_g \Big\}_{g \in G},
\]
given by the coaction by $G$ on $A \times_X P$.
A representation of $X$ that promotes to a representation of $\S\C X$ will be called a \emph{strong covariant representation of $X$}. Equivalently, a representation $t$ of $X$ is strong covariant if and only if $t_\ast$ vanishes on $\I_{\scv,e}$.
It follows that $A \times_X P$ is the full C*-algebra of $\S\C X$; see \cite{DKKLL22}.
We define the \emph{reduced strong covariant C*-algebra} $A \times_{X, \la} P$ to be the reduced C*-algebra of $\S\C X$.
By definition of $\I_{\scv,e}$ we get that
\[
\ker\la_\ast\cap[\T(X)]_{e_P}\subseteq\I_{\scv,e},
\]
and hence it follows that a strong covariant representation is Fock covariant.
By \cite[Theorem~3.10]{Seh19}, the C*-algebra $A$ embeds in $A\times_X P$, and a strong covariant representation $t$ of $X$ is injective if and only if the induced $*$-representation of $A\times_X P$ is injective on the fixed point algebra $[A\times_X P]_{e_P}$.

For a product system $X$ we write $\T_\la(X)^+$ for the \emph{tensor algebra of $X$}, i.e., for the norm-closed subalgebra generated by $\{\la_p(X_p)\}_{p \in P}$ inside $\T_\la(X)$.
In \cite[Theorem~5.1]{Seh22} Sehnem proved that the C*-envelope $\cenv(\T_\la(X)^+)$ admits a normal coaction by $G$ and thus there exists a canonical $*$-isomorphism
\begin{equation}\label{Eq:Sehcenv}
\cenv(\T_\la(X)^+) \simeq A \times_{X, \la} P.
\end{equation}

We will also need some further properties of the representations $\Phi_F$.
Let $F\subseteq G$ be finite.
As noted above, $X_{F}$ is reducing for $[\T_\la(X)]_{e_P}$, and $\Phi_F(d)|_{X_F}$ is just the restriction of $d$ to $X_F$ for every $d\in[\T_\la(X)]_{e_P}$.
To ease notation we write
\[
\pi_F \colon [\T_\la(X)]_{e_P} \to \L(X_F);\ d \mapsto \Phi_F(d)|_{X_F}=d|_{X_F}.
\]
If $F\subseteq F'$, then $X_{F'}\subseteq X_F$, and thus $\pi_{F'}(d)=\pi_F(d)|_{X_{F'}}$.
Hence
\begin{equation}\label{Eq:monot}
\|\pi_{F'}(d)\|\leq\|\pi_F(d)\| \foral d\in[\T_\la(X)]_{e_P} \text{ and finite } F\subseteq F'\subseteq G.
\end{equation}
In particular, the limit $\lim_F\|\pi_F(d)\|$ over the finite subsets of $G$ exists for every $d\in[\T_\la(X)]_{e_P}$, and 
\[
b\in[\T(X)]_{e_P} \text{ lies in } \I_{\scv,e} \text{ if and only if } \lim_F\|\pi_F(\la_\ast(b))\|=0.
\]
Since $\T_\la(X)$ is a C*-cover of $\T_\la(X)^+$, by \eqref{Eq:Sehcenv} we obtain a canonical $*$-epimorphism
\[
q_{{\rm sc},\la}\colon\T_\la(X)\to A\times_{X,\la}P.
\]
Consider also the canonical $*$-epimorphism $ \Psi \colon A\times_X P\to A\times_{X,\la}P$. Since all these maps are canonical, it follows that
\[
q_{{\rm sc},\la}\circ\la_\ast= \Psi \circ q_{{\rm sc}}.
\]
Since $q_{{\rm sc},\la}$ is completely isometric on $\T_\la(X)^+$, it is injective on $\la_{e_P}(A)$.
Moreover, we have
\begin{equation}\label{Eq:Jker}
\{d\in[\T_\la(X)]_{e_P} \mid \lim_F\|\pi_F(d)\|=0\}\subseteq\ker q_{{\rm sc},\la}.
\end{equation}
Indeed, it is easy to see that every $d\in[\T_\la(X)]_{e_P}$ is of the form $d=\la_\ast(b)$ for some $b\in[\T(X)]_{e_P}$.
If in addition $\lim_F\|\pi_F(d)\|=0$, then $b\in\I_{\scv,e}$, and thus
\[
q_{{\rm sc},\la}(d)=q_{{\rm sc},\la}(\la_\ast(b))=\Psi \circ q_{{\rm sc}}(b)=0.
\]
We note that although only non-degenerate product systems in the sense of
Fowler \cite{Fow02} are considered in \cite{Seh19, Seh22}, the arguments apply in our setting as well. See also \cite[Remark 2.2]{DKKLL22}.

\section{The Fock C*-algebra of the induced product system}\label{S:Fock}

In this section we show that the Fock C*-algebra of the induced product system is canonically $*$-isomorphic to the reduced crossed product of the Fock C*-algebra. This answers in full generality a question that Kakariadis and the author raised in \cite[Introduction, comments prior to Theorem~B]{KP25}, and this was not known for product systems over unital subsemigroups of discrete groups even when the generalised gauge action is by a discrete group.

Let $P$ be a unital subsemigroup of a discrete group $G$ and let $X$ be a product system over $P$.
A \emph{generalised gauge action} of a locally compact Hausdorff group $\fH$ on $\T_\la(X)$ is a point-norm continuous homomorphism $\al\colon\fH\to\Aut(\T_\la(X))$ such that
\[
\al_\fh(\la_p(X_p))=\la_p(X_p) \foral p\in P \text{ and } \fh\in\fH.
\]
For $p\in P$ we set
\[
X_p\rtimes_{\al,\la}\fH:=[C_c(\fH,\la_p(X_p))]\subseteq\T_\la(X)\rtimes_{\al,\la}\fH,
\]
and
\[
X\rtimes_{\al,\la}\fH:=\{X_p\rtimes_{\al,\la}\fH\}_{p\in P}.
\]
Since $\la$ is a representation of $X$ and every $\la_p(X_p)$ is $\al$-invariant, the convolution formulas show that 
\[
(X_p\rtimes_{\al,\la}\fH)(X_q\rtimes_{\al,\la}\fH)\subseteq X_{pq}\rtimes_{\al,\la}\fH \qand (X_p\rtimes_{\al,\la}\fH)^*(X_{pq}\rtimes_{\al,\la}\fH)\subseteq X_q\rtimes_{\al,\la}\fH \foral p,q\in P,
\]
while $X_{e_P}\rtimes_{\al,\la}\fH$ is a C*-algebra, and hence $X\rtimes_{\al,\la}\fH$ is a product system over $P$; see also \cite[Section~3]{Kat20} and \cite[Lemma~7.11]{KR19}.
We identify $A$ with $\la_{e_P}(A)$, so that $X_{e_P}\rtimes_{\al,\la}\fH$ is the reduced crossed product $A\rtimes_{\al,\la}\fH$.

For the rest of the paper we fix a product system $X$ over a unital subsemigroup $P$ of a discrete group $G$, and a generalised gauge action $\al$ of a locally compact Hausdorff group $\fH$ on $\T_\la(X)$.
We will need the following proposition.

\begin{proposition}\label{P:gen}
The norm-closed algebra generated by $\bigcup_{p\in P}X_p\rtimes_{\al,\la}\fH$ coincides with the norm-closed subalgebra of $\T_\la(X)\rtimes_{\al,\la}\fH$ generated by $C_c(\fH,\T_\la(X)^+)$.
Moreover, the C*-algebra generated by $\bigcup_{p\in P}X_p\rtimes_{\al,\la}\fH$ is $\T_\la(X)\rtimes_{\al,\la}\fH$.
\end{proposition}

\begin{proof}
Let $\C^+$ denote the norm-closed algebra generated by $\bigcup_{p\in P}X_p\rtimes_{\al,\la}\fH$, i.e., the norm-closed algebra generated by $\bigcup_{p\in P}C_c(\fH,\la_p(X_p))$. Since $\la_p(X_p)\subseteq\T_\la(X)^+$ for every $p\in P$, the algebra $\C^+$ is contained in the norm-closed algebra generated by $C_c(\fH,\T_\la(X)^+)$. For the converse inclusion it suffices to show that $C_c(\fH,\T_\la(X)^+)\subseteq\C^+$.

For $f\in C_c(\fH)$ and $b\in\T_\la(X)$ we write $f\otimes b$ for the function $\fg\mapsto f(\fg)b$ in $C_c(\fH,\T_\la(X))$, and we set
\[
\B^+:=\{b\in\T_\la(X)\mid f\otimes b\in\C^+ \foral f\in C_c(\fH)\}.
\]
It follows that $\la_p(X_p)\subseteq\B^+$ for every $p\in P$. We will show that $\B^+$ is a norm-closed subalgebra of $\T_\la(X)$ and hence it contains $\T_\la(X)^+$.
It is clear that $\B^+$ is a linear subspace of $\T_\la(X)$. Let $(b_i)_i\subseteq\B^+$ be a net such that $b_i\xrightarrow{i}b$ in norm. Since the norm of $\T_\la(X)\rtimes_{\al,\la}\fH$ is dominated by the $L^1$-norm, we have
\[
\|f\otimes b_i-f\otimes b\|\leq\|f\|_1\cdot\|b_i-b\|\xrightarrow{i}0 \foral f\in C_c(\fH),
\]
and hence $b\in\B^+$ as $\C^+$ is norm-closed.
Now let $x,y\in\B^+$ and $f\in C_c(\fH)$. Fix a compact neighbourhood $V$ of $e_\fH$. For every neighbourhood $U\subseteq V$ of $e_\fH$ choose a non-negative function $e_U\in C_c(\fH)$ with $\supp e_U\subseteq U$ and $\int_\fH e_U\,d\mu=1$, and direct these neighbourhoods by reverse inclusion. Since $e_U\otimes x$ and $f\otimes y$ lie in $\C^+$, so does their convolution product 
\[
F_U(\fg'):=((e_U\otimes x)\ast(f\otimes y))(\fg')=\int_{\fH}e_U(\fg)f(\fg^{-1}\fg')\,x\al_\fg(y)\,d\mu(\fg) \foral \fg'\in\fH.
\]
Every $F_U$ is supported in the compact set $C:=V\cdot\supp f$, and since $\int_\fH e_U\,d\mu=1$ we have
\[
\|F_U(\fg')-f(\fg')xy\|
\leq
\|x\|\cdot\sup_{\fg\in U}\big(\|f\|_\infty\cdot\|\al_\fg(y)-y\|+\|y\|\cdot|f(\fg^{-1}\fg')-f(\fg')|\big)
\foral \fg'\in\fH.
\]
By point-norm continuity of $\al$ and uniform continuity of the compactly supported function $f$ with respect to left translation, the right-hand side converges to zero uniformly in $\fg'\in\fH$. As $f\otimes xy$ is also supported in $C$, we get
\[
\|F_U-f\otimes xy\|
\leq
\|F_U-f\otimes xy\|_1
\leq
\mu(C)\cdot\sup_{\fg'\in\fH}\|F_U(\fg')-f(\fg')xy\|
\xrightarrow{U}0,
\]
and hence $f\otimes xy\in\C^+$ as $\C^+$ is norm-closed. Since $f$ was arbitrary we get $xy\in\B^+$.

Consequently, we have that $f\otimes b\in\C^+$ for all $f\in C_c(\fH)$ and $b\in\T_\la(X)^+$. By \cite[Lemma~1.87]{Wil07} the linear span of such functions is dense in $C_c(\fH,\T_\la(X)^+)$ with respect to the $L^1$-norm, and thus in the norm of $\T_\la(X)\rtimes_{\al,\la}\fH$. Since $\C^+$ is norm-closed we get $C_c(\fH,\T_\la(X)^+)\subseteq\C^+$, and this proves the first assertion.

Similarly, since $\la_p(X_p)$ is $\al$-invariant, the C*-algebra $\C$ generated by $\bigcup_{p\in P}X_p\rtimes_{\al,\la}\fH$ contains $C_c(\fH,\la_p(X_p)^*)$ for every $p\in P$, and the same arguments with $\C$ in place of $\C^+$ yield $C_c(\fH,\T_\la(X))\subseteq\C$. Since $C_c(\fH,\T_\la(X))$ is dense in $\T_\la(X)\rtimes_{\al,\la}\fH$, we conclude that $\C=\T_\la(X)\rtimes_{\al,\la}\fH$, as required.
\end{proof}

The action $\al$ induces an action on the Fock space as follows.
Since each $\la_p$ is isometric, for $\fh\in\fH$ and $p\in P$ there is a unique linear isometric bijection $\ga_\fh$ of $X_p$ such that
\[
\la_p(\ga_\fh(\xi_p))=\al_\fh(\la_p(\xi_p)) \foral \xi_p\in X_p.
\]
In particular, we have $\ga_\fh|_A=\al_\fh|_A$ under our identification of $A$ with $\la_{e_P}(A)$.
Since $\la$ is an injective representation of $X$ and $\al_\fh$ is a $*$-automorphism, we have
\[
\ga_\fh(\xi_p\cdot\xi_q)=\ga_\fh(\xi_p)\cdot\ga_\fh(\xi_q)
\qand
\ga_\fh(\xi_p^*\cdot\xi_{pq})=\ga_\fh(\xi_p)^*\cdot\ga_\fh(\xi_{pq}) \foral \xi_p\in X_p, \xi_q\in X_q, \xi_{pq}\in X_{pq}.
\]
Hence every $X_q$ is an $\fH$-equivariant Hilbert $A$-module with respect to $\ga$, and so is $\F X$ with respect to $\ga_\fh:=\sum^\oplus_{r\in P}\ga_\fh|_{X_r}$.
For $\xi_p\in X_p$ and $\eta_r\in X_r$ we have
\[
\ga_\fh(\la_p(\xi_p)\ga_{\fh^{-1}}(\eta_r))
=\ga_\fh(\xi_p \cdot \ga_{\fh^{-1}}(\eta_r))
=\ga_\fh(\xi_p)\cdot\eta_r
=\al_\fh(\la_p(\xi_p))\eta_r,
\]
and therefore we obtain
\begin{equation}\label{Eq:spatial}
\ga_\fh\circ d\circ\ga_{\fh^{-1}} = \al_\fh(d) \foral d\in\T_\la(X) \text{ and } \fh\in\fH,
\end{equation}
and thus the action $\al$ is the restriction of the action $\ad\ga$ on $\L(\F X)$ to $\T_\la(X)$.

We consider the identity representation
\[
\io^\rtimes\colon X\rtimes_{\al,\la}\fH\to\T_\la(X)\rtimes_{\al,\la}\fH,
\]
and we write $\wh{t}^\rtimes$ for the universal representation of $X\rtimes_{\al,\la}\fH$, and $\la^\rtimes$ for its Fock representation on $\F(X\rtimes_{\al,\la}\fH)$.

The next lemma identifies, as Hilbert modules, the Fock space of $X\rtimes_{\al,\la}\fH$ with the reduced crossed product of the Fock space of $X$ with respect to the induced action described above.

\begin{lemma}\label{L:Fockcp}
With the aforementioned notation, the following hold:
\begin{enumerate}
\item For every $r\in P$, the map $\ze\mapsto\la_r\circ\ze$ for $\ze\in C_c(\fH,X_r)$ extends to a unitary from $X_r\rtimes_{\ga,\la}\fH$ onto $X_r\rtimes_{\al,\la}\fH$.
Consequently, $\F X\rtimes_{\ga,\la}\fH$ is canonically unitarily equivalent to $\F(X\rtimes_{\al,\la}\fH)$.
\item There exists a faithful $*$-representation $\La\colon\T_\la(X)\rtimes_{\al,\la}\fH\to\L(\F(X\rtimes_{\al,\la}\fH))$ that satisfies $\La\circ\io^\rtimes_\ast=\la^\rtimes_\ast$.
\end{enumerate}
\end{lemma}

\begin{proof}
We begin by proving item (i). Recall that $\ga_\fh|_A=\al_\fh|_A$.
For $\ze,\eta\in C_c(\fH,X_r)$ we have
\begin{align*}
((\la_r\circ\ze)^*\ast(\la_r\circ\eta))(\fg')
& =
\int_{\fH}\De(\fg)^{-1}\al_{\fg}\big(\la_r(\ze(\fg^{-1}))^*\la_r(\eta(\fg^{-1}\fg'))\big)\,d\mu(\fg) \\
& =
\int_{\fH}\al_{\fg^{-1}}\big(\la_r(\ze(\fg))^*\la_r(\eta(\fg\fg'))\big)\,d\mu(\fg) \\
& =
\int_{\fH}\al_{\fg^{-1}}(\la_{e_P}(\sca{\ze(\fg),\eta(\fg\fg')}))\,d\mu(\fg)\\
&=
\int_{\fH}\la_{e_P}\big(\al_{\fg^{-1}}(\sca{\ze(\fg),\eta(\fg\fg')})\big)\,d\mu(\fg)
=
\la_{e_P}\big(\sca{\ze,\eta}(\fg')\big).
\end{align*}
Similarly, for $\ze\in C_c(\fH,X_r)$ and $f\in C_c(\fH,A)$ we have
\begin{align*}
((\la_r\circ\ze)\ast(\la_{e_P}\circ f))(\fg')
& =
\int_{\fH}\la_r(\ze(\fg))\al_{\fg}\big(\la_{e_P}(f(\fg^{-1}\fg'))\big)\,d\mu(\fg) \\
& =
\int_{\fH}\la_r(\ze(\fg))\la_{e_P}\big(\al_\fg (f(\fg^{-1}\fg'))\big)\,d\mu(\fg) \\
& =
\int_{\fH}\la_r\big(\ze(\fg)\al_{\fg}(f(\fg^{-1}\fg'))\big)\,d\mu(\fg)
=
\la_r\big((\ze\cdot f)(\fg')\big).
\end{align*}
Hence the map $\ze\mapsto\la_r\circ\ze$ preserves the inner products and the right module actions.
Since its range $C_c(\fH,\la_r(X_r))$ is dense in $X_r\rtimes_{\al,\la}\fH$, it extends to a unitary by \cite[Theorem~3.5]{Lan95}.
The second assertion follows from Proposition~\ref{P:modsub}~(ii).

For item (ii), note that $\T_\la(X)\subseteq\L(\F X)$ is $(\ad\ga)$-invariant, and that $\ad\ga$ is point-norm continuous on $\T_\la(X)$, as by \eqref{Eq:spatial} it coincides with $\al$ there.
Hence Proposition~\ref{P:modrep} applies and we obtain a faithful $*$-representation
\[
\ka\colon \T_\la(X)\rtimes_{\al,\la}\fH \to \L(\F X\rtimes_{\ga,\la}\fH).
\]
Let also 
\[
U\colon\F X\rtimes_{\ga,\la}\fH\to\F(X\rtimes_{\al,\la}\fH)
\]
be the unitary obtained in item (i). Define
\[
\La\colon\T_\la(X)\rtimes_{\al,\la}\fH\to\L(\F(X\rtimes_{\al,\la}\fH));\ d\mapsto U\ka(d)U^*.
\]
Let $p,r\in P$, let $y\in C_c(\fH,\la_p(X_p))$ and $\ze\in C_c(\fH,X_r)$, and write $y(\fg)=\la_p(\xi(\fg))$ for a continuous function $\xi\colon\fH\to X_p$.
Then $\ka(y)\ze\in C_c(\fH,X_{pr})$, and since $\la$ is a representation of $X$ we have
\begin{align*}
(\la_{pr} \circ (\ka(y)\ze))(\fg')
&=
\la_{pr}\Big(\int_{\fH} \xi(\fg)\cdot\ga_\fg(\ze(\fg^{-1}\fg'))\,d\mu(\fg) \Big)\\
&=
\int_{\fH} \la_p(\xi(\fg)) \la_r(\ga_\fg(\ze(\fg^{-1}\fg')))\,d\mu(\fg)\\
&=
\int_{\fH}y(\fg)\al_\fg\big(\la_r(\ze(\fg^{-1}\fg'))\big)\,d\mu(\fg)
=
(y\ast(\la_r\circ\ze))(\fg').
\end{align*}
We thus obtain 
\[
U\ka(y)\ze= \la_{pr}\circ(\ka(y)\ze) = y\ast(\la_r\circ\ze)=y\ast U\ze=\la^\rtimes_p(y)U\ze.
\]
Now pick $\eta\in C_c(\fH,\la_r(X_r))$.
Then $\eta=U\ze$ for some $\ze\in C_c(\fH,X_r)$, and therefore
\[
\La(y)\eta=U\ka(y)U^*U \ze=U\ka(y)\ze=\la^\rtimes_p(y)U\ze=\la^\rtimes_p(y)\eta.
\]
Hence the bounded operators $\La(y)$ and $\la^\rtimes_p(y)$ agree on $C_c(\fH,\la_r(X_r))$, which is dense in $X_r\rtimes_{\al,\la}\fH$.
Since $p,r$ and $y$ were arbitrary, by density we obtain 
\[
\La( \io^\rtimes_p(y))=\la^\rtimes_p(y) \foral y \in X_p\rtimes_{\al,\la}\fH \text{ and } p\in P.
\]
Consequently, we get that
$\La\circ \io^\rtimes_\ast=\la^\rtimes_\ast$,
as required.
\end{proof}

We are now ready to prove the first main result of this work concerning the Fock C*-algebra of the induced product system.

\begin{theorem}\label{T:Fock}
Let $X$ be a product system over a unital subsemigroup $P$ of a discrete group $G$, and let $\al$ be a generalised gauge action of a locally compact Hausdorff group $\fH$ on $\T_\la(X)$.
Then there exists a canonical $*$-isomorphism
\[
\T_\la(X \rtimes_{\al, \la} \fH) \simeq \T_\la(X) \rtimes_{\al, \la} \fH.
\]
In particular, the identity representation $\io^\rtimes \colon X\rtimes_{\al,\la}\fH \to \T_\la(X)\rtimes_{\al,\la}\fH$ is injective, Fock covariant, and admits a normal coaction by $G$.
\end{theorem}

\begin{proof}
By Proposition~\ref{P:gen} and Lemma~\ref{L:Fockcp}~(ii), we have
\begin{align*}
\La(\T_\la(X)\rtimes_{\al,\la}\fH)
&=
\La(\ca(\io^\rtimes))
= \La \circ\io^\rtimes_\ast(\T(X\rtimes_{\al,\la}\fH))\\
&=
\la^\rtimes_\ast(\T(X\rtimes_{\al,\la}\fH))
=
\ca(\la^\rtimes)
=
\T_\la(X\rtimes_{\al,\la}\fH).
\end{align*}
Hence $\La$ is a $*$-isomorphism onto $\T_\la(X\rtimes_{\al,\la}\fH)$ with $\La\circ\io^\rtimes=\la^\rtimes$.
The Fock representation $\la^\rtimes$ is injective and Fock covariant, and admits a normal coaction by $G$.
As $\io^\rtimes_\ast=\La^{-1}\circ\la^\rtimes_\ast$, the same holds for $\io^\rtimes$.
\end{proof}

We obtain the following corollary.

\begin{corollary}\label{C:tenscros}
Let $X$ be a product system over a unital subsemigroup $P$ of a discrete group $G$, and let $\al$ be a generalised gauge action of a locally compact Hausdorff group $\fH$ on $\T_\la(X)$.
Then there exists a canonical completely isometric isomorphism
\[
\T_\la(X \rtimes_{\al, \la} \fH)^+ \simeq \T_\la(X)^+ \rtimes_{\al, \la} \fH.
\]
\end{corollary}

\begin{proof}
By Remark~\ref{R:cai}, the reduced crossed product $\T_\la(X)^+\rtimes_{\al,\la}\fH$ is completely isometrically isomorphic to the norm-closed subalgebra of $\T_\la(X)\rtimes_{\al,\la}\fH$ generated by $C_c(\fH,\T_\la(X)^+)$. By Proposition~\ref{P:gen} this is the norm-closed algebra generated by $\bigcup_{p\in P}X_p\rtimes_{\al,\la}\fH$, which the $*$-isomorphism of Theorem~\ref{T:Fock} maps onto $\T_\la(X\rtimes_{\al,\la}\fH)^+$.
\end{proof}

\section{The reduced Hao--Ng isomorphism}\label{S:HaoNg}

In this section we compare the strong covariance relations of $X$ and $X\rtimes_{\al,\la}\fH$, and we resolve the reduced Hao--Ng isomorphism problem for product systems in full generality.
We keep the notation of the previous sections.
When we apply the constructions of Section~\ref{S:prelim} to the product system $X\rtimes_{\al,\la}\fH$, we make use of the symbol $\rtimes$.
Thus we write $\vphi^\rtimes_q$, $I^\rtimes_{r^{-1}K_{\{r,g\}}}$ and $I^\rtimes_{r^{-1}(r\vee F)}$ for the left actions and the corresponding ideals of $A\rtimes_{\al,\la}\fH$, and we write $(X\rtimes_{\al,\la}\fH)_F$, $\pi^\rtimes_F$ and $\I^\rtimes_{\scv,e}$ for the corresponding objects of $X\rtimes_{\al,\la}\fH$.

Note that $\ker q_{{\rm sc},\la}$ is the Shilov boundary ideal of $\T_\la(X)^+$ in $\T_\la(X)$. We have that $\al_\fh$ maps $\T_\la(X)^+$ onto itself for all $\fh\in\fH$, and so the quotient map $\T_\la(X)\to\T_\la(X)/\al_\fh(\ker q_{{\rm sc},\la})$ is completely isometric on $\T_\la(X)^+$. By maximality we get $\al_\fh(\ker q_{{\rm sc},\la})\subseteq\ker q_{{\rm sc},\la}$, and equality follows by applying this to $\fh^{-1}$. Hence there is a $*$-automorphism $\dot\al_\fh$ of $A\times_{X,\la}P$ with 
\[
\dot\al_\fh\circ q_{{\rm sc},\la}=q_{{\rm sc},\la}\circ\al_\fh.
\]
Then $\dot\al$ is a group homomorphism, and it is point-norm continuous since 
\[
\|\dot\al_\fh(q_{{\rm sc},\la}(d))-q_{{\rm sc},\la}(d)\|\leq\|\al_\fh(d)-d\| \foral d\in\T_\la(X).
\]
In particular, we get that $q_{{\rm sc},\la}$ is an $\fH$-equivariant $*$-epimorphism, and hence it induces a $*$-epimorphism
\[
q_{{\rm sc},\la}\rtimes\fH\colon\T_\la(X)\rtimes_{\al,\la}\fH\to(A\times_{X,\la}P)\rtimes_{\dot\al,\la}\fH
\]
that extends the map $f\mapsto q_{{\rm sc},\la}\circ f$ on $C_c(\fH,\T_\la(X))$.

If $a\cdot X_q=(0)$, then $\al_\fh(a)\cdot X_q=\ga_\fh(a\cdot\ga_{\fh^{-1}}(X_q))=(0)$, and hence 
\[
\al_\fh(\ker\vphi_q)=\ker\vphi_q \foral q\in P \text{ and }\fh \in \fH.
\]
Thus, the ideal $I_{r^{-1}K_{\{r,g\}}}$ is $\al$-invariant for all $r\in P$ and $g\in G$, as an intersection of $\al$-invariant ideals.
Similarly, we get that the ideal $I_{r^{-1}(r\vee F)}$ is $\al$-invariant for all $r\in P$ and for any finite $F\subseteq G$. In particular, we get that $\ga_\fh(X_F)=X_F$ for all $\fh \in \fH$, and hence $X_F$ is an $\fH$-equivariant Hilbert $A$-module with respect to the restriction of $\ga$.

Since $\al$ preserves the spectral subspaces of $\T_\la(X)$, we also have 
\[
\al_\fh([\T_\la(X)]_{e_P})=[\T_\la(X)]_{e_P}.
\]
Thus by injectivity of the reduced crossed product, we get
\[
[\T_\la(X)]_{e_P}\rtimes_{\al,\la}\fH=[C_c(\fH,[\T_\la(X)]_{e_P})]\subseteq\T_\la(X)\rtimes_{\al,\la}\fH.
\]
By \eqref{Eq:spatial} we have
\[
\pi_F(\al_\fh(d))=\ga_\fh\circ\pi_F(d)\circ\ga_{\fh^{-1}}|_{X_F} \foral d\in[\T_\la(X)]_{e_P} \text{ and } \fh \in \fH.
\]
Hence $\pi_F([\T_\la(X)]_{e_P})$ is an $(\ad\ga)$-invariant C*-subalgebra of $\L(X_F)$ on which $\ad\ga$ is point-norm continuous.
By Proposition~\ref{P:modrep} we obtain a faithful $*$-representation
\begin{equation}\label{Eq:kappaf}
\ka_F\colon\pi_F([\T_\la(X)]_{e_P})\rtimes_{\ad\ga,\la}\fH\to\L(X_F\rtimes_{\ga,\la}\fH).
\end{equation}
As $\pi_F$ is $\fH$-equivariant, we also obtain a $*$-epimorphism
\[
\pi_F\rtimes\fH\colon[\T_\la(X)]_{e_P}\rtimes_{\al,\la}\fH\to\pi_F([\T_\la(X)]_{e_P})\rtimes_{\ad\ga,\la}\fH.
\]

The following lemma shows that $q_{{\rm sc},\la}\rtimes\fH$ factors through each $\pi_F\rtimes\fH$ on the fixed point algebra.

\begin{lemma}\label{L:factor}
Let $F\subseteq G$ be finite.
With the aforementioned notation, we have
\[
\|(q_{{\rm sc},\la}\rtimes\fH)(d)\|\leq\|(\pi_F\rtimes\fH)(d)\| \foral d\in[\T_\la(X)]_{e_P}\rtimes_{\al,\la}\fH.
\]
\end{lemma}

\begin{proof}
Let $c\in[\T_\la(X)]_{e_P}$ with $\pi_F(c)=0$.
By \eqref{Eq:monot} we have $\pi_{F'}(c)=0$ for every finite $F'\subseteq G$ that contains $F$.
Since these sets are cofinal, we get $\lim_{F'}\|\pi_{F'}(c)\|=0$, and hence $q_{{\rm sc},\la}(c)=0$ by \eqref{Eq:Jker}.
Therefore $\ker\pi_F\subseteq\ker q_{{\rm sc},\la}$, and so the map
\[
\theta_F\colon\pi_F([\T_\la(X)]_{e_P})\to q_{{\rm sc},\la}([\T_\la(X)]_{e_P});\ \pi_F(c)\mapsto q_{{\rm sc},\la}(c),
\]
is a well-defined $*$-epimorphism such that $\theta_F\circ\pi_F=q_{{\rm sc},\la}|_{[\T_\la(X)]_{e_P}}$.

Since $[\T_\la(X)]_{e_P}$ is $\al$-invariant and $\dot\al_\fh\circ q_{{\rm sc},\la}=q_{{\rm sc},\la}\circ\al_\fh$, the C*-subalgebra $q_{{\rm sc},\la}([\T_\la(X)]_{e_P})$ of $A\times_{X,\la}P$ is $\dot\al$-invariant, and $\dot\al$ is point-norm continuous on it.
Moreover, for $c\in[\T_\la(X)]_{e_P}$ and $\fh\in\fH$ we have
\[
\theta_F\big((\ad\ga)_\fh(\pi_F(c))\big)
=\theta_F(\pi_F(\al_\fh(c)))
=q_{{\rm sc},\la}(\al_\fh(c))
=\dot\al_\fh(q_{{\rm sc},\la}(c))
=\dot\al_\fh\big(\theta_F(\pi_F(c))\big),
\]
and so $\theta_F$ is $\fH$-equivariant.
Hence it induces a $*$-homomorphism
\[
\theta_F\rtimes\fH\colon\pi_F([\T_\la(X)]_{e_P})\rtimes_{\ad\ga,\la}\fH\to q_{{\rm sc},\la}([\T_\la(X)]_{e_P})\rtimes_{\dot\al,\la}\fH.
\]
In addition, the inclusion $j\colon q_{{\rm sc},\la}([\T_\la(X)]_{e_P})\hookrightarrow A\times_{X,\la}P$ is injective and $\fH$-equivariant, and thus $j\rtimes\fH$ is an injective $*$-homomorphism into $(A\times_{X,\la}P)\rtimes_{\dot\al,\la}\fH$.

For $f\in C_c(\fH,[\T_\la(X)]_{e_P})$ we have
\[
\big((j\rtimes\fH)\circ(\theta_F\rtimes\fH)\circ(\pi_F\rtimes\fH)\big)(f)
=j\circ\theta_F\circ\pi_F\circ f
=q_{{\rm sc},\la}\circ f
=(q_{{\rm sc},\la}\rtimes\fH)(f),
\]
and so by continuity
\[
(q_{{\rm sc},\la}\rtimes\fH)(d)=(j\rtimes\fH)\big((\theta_F\rtimes\fH)((\pi_F\rtimes\fH)(d))\big) \foral d\in[\T_\la(X)]_{e_P}\rtimes_{\al,\la}\fH.
\]
We conclude that
\[
\|(q_{{\rm sc},\la}\rtimes\fH)(d)\|
=\|(\theta_F\rtimes\fH)((\pi_F\rtimes\fH)(d))\|
\leq\|(\pi_F\rtimes\fH)(d)\|,
\]
as required.
\end{proof}

Next we compare the representations $\pi_F$ with the representations $\pi^\rtimes_F$.

\begin{lemma}\label{L:compare}
Let $F\subseteq G$ be finite and let $b$ be in $[\T(X\rtimes_{\al,\la}\fH)]_{e_P}$.
With the aforementioned notation, we have $\io^\rtimes_\ast(b)\in[\T_\la(X)]_{e_P}\rtimes_{\al,\la}\fH$ and
\[
\|(\pi_F\rtimes\fH)(\io^\rtimes_\ast(b))\|\leq\|\pi^\rtimes_F(\la^\rtimes_\ast(b))\|.
\]
\end{lemma}

\begin{proof}
By density, the fixed point algebra $[\T(X\rtimes_{\al,\la}\fH)]_{e_P}$ is the closed linear span of the elements
\[
\left(\wh{t}^\rtimes_{p_1}(f_{p_1})^*\right)^{\eps}\wh{t}^\rtimes_{q_1}(h_{q_1})\cdots\wh{t}^\rtimes_{p_n}(f_{p_n})^*\wh{t}^\rtimes_{q_n}(h_{q_n})^{\eps'}
\text{ such that }
p_1^{-\eps}q_1\cdots p_n^{-1}q_n^{\eps'}=e_P,
\]
where $\eps,\eps'\in\{0,1\}$, $n \in \bN$, $f_{p_i}\in C_c(\fH,\la_{p_i}(X_{p_i}))$ and $h_{q_i}\in C_c(\fH,\la_{q_i}(X_{q_i}))$.
The image of such an element under $\io^\rtimes_\ast$ is the convolution product $(f_{p_1}^*)^{\eps}\ast h_{q_1}\ast\cdots\ast f_{p_n}^*\ast h_{q_n}^{\eps'}$.
As every $\la_p(X_p)$ is $\al$-invariant, an induction on $n$ shows that it lies in
\[
C_c\left(\fH,\left[\left(\la_{p_1}(X_{p_1})^*\right)^{\eps}\la_{q_1}(X_{q_1})\cdots\la_{p_n}(X_{p_n})^*\la_{q_n}(X_{q_n})^{\eps'}\right]\right)\subseteq C_c(\fH,[\T_\la(X)]_{e_P}).
\]
Hence $\io^\rtimes_\ast(b)\in[\T_\la(X)]_{e_P}\rtimes_{\al,\la}\fH$.
Moreover, $\la^\rtimes_\ast$ maps the elements above into $[\T_\la(X\rtimes_{\al,\la}\fH)]_{e_P}$, and so $\la^\rtimes_\ast(b)\in[\T_\la(X\rtimes_{\al,\la}\fH)]_{e_P}$.

Next we compare the ideals.
Let $q$ be in $P$.
If $f\in C_c(\fH,\la_{e_P}(\ker\vphi_q))$ and $h\in C_c(\fH,\la_q(X_q))$, then the integrand in
\[
(f\ast h)(\fg')=\int_{\fH}f(\fg)\al_\fg(h(\fg^{-1}\fg'))\,d\mu(\fg)
\]
vanishes, since $\al_\fg(h(\fg^{-1}\fg'))\in\la_q(X_q)$.
As $C_c(\fH,\la_q(X_q))$ is dense in $X_q\rtimes_{\al,\la}\fH$, we get 
\[
(\ker\vphi_q)\rtimes_{\al,\la}\fH\subseteq\ker\vphi^\rtimes_q.
\]
Now fix $r\in P$ and $g\in G$.
If $K_{\{r,g\}}=\mt$ or $r\in K_{\{r,g\}}$, then by definition we have
\[
I_{r^{-1}K_{\{r,g\}}}\rtimes_{\al,\la}\fH=A\rtimes_{\al,\la}\fH=I^\rtimes_{r^{-1}K_{\{r,g\}}}.
\]
If $K_{\{r,g\}}\neq\mt$ and $r\notin K_{\{r,g\}}$, then $I_{r^{-1}K_{\{r,g\}}}$ is an $\al$-invariant ideal of $A$, and it is contained in $\ker\vphi_{r^{-1}s}$ for every $s\in K_{\{r,g\}}$.
By injectivity of the reduced crossed product we get
\[
I_{r^{-1}K_{\{r,g\}}}\rtimes_{\al,\la}\fH
\subseteq
(\ker\vphi_{r^{-1}s})\rtimes_{\al,\la}\fH
\subseteq
\ker\vphi^\rtimes_{r^{-1}s}
\foral s\in K_{\{r,g\}},
\]
and by taking the intersection over $s\in K_{\{r,g\}}$ we get
\[
I_{r^{-1}K_{\{r,g\}}}\rtimes_{\al,\la}\fH
\subseteq
\bigcap_{s\in K_{\{r,g\}}}\ker\vphi^\rtimes_{r^{-1}s}
=
I^\rtimes_{r^{-1}K_{\{r,g\}}}.
\]
In either case we have $I_{r^{-1}K_{\{r,g\}}}\rtimes_{\al,\la}\fH\subseteq I^\rtimes_{r^{-1}K_{\{r,g\}}}$ for all $r\in P$ and $g\in G$.
Finally, since $I_{r^{-1}(r\vee F)}\subseteq I_{r^{-1}K_{\{r,g\}}}$ for every $g\in F$ and $r\in P$, another application of injectivity of the reduced crossed product gives
\[
I_{r^{-1}(r\vee F)}\rtimes_{\al,\la}\fH
\subseteq
I_{r^{-1}K_{\{r,g\}}}\rtimes_{\al,\la}\fH
\subseteq
I^\rtimes_{r^{-1}K_{\{r,g\}}}
\foral g\in F \text{ and } r\in P,
\]
and by taking the intersection over $g\in F$ we conclude that
\begin{equation}\label{Eq:idealcp}
I_{r^{-1}(r\vee F)}\rtimes_{\al,\la}\fH\subseteq I^\rtimes_{r^{-1}(r\vee F)} \foral r\in P.
\end{equation}

Next we compare the submodules.
Let $U\colon\F X\rtimes_{\ga,\la}\fH\to\F(X\rtimes_{\al,\la}\fH)$ be the unitary of Lemma~\ref{L:Fockcp}~(i).
As $U$ is a unitary of right Hilbert $(A\rtimes_{\al,\la}\fH)$-modules that maps $X_r\rtimes_{\ga,\la}\fH$ onto $X_r\rtimes_{\al,\la}\fH$, Proposition~\ref{P:modsub}~(iii) and \eqref{Eq:idealcp} give
\[
U\big((X_r I_{r^{-1}(r\vee F)})\rtimes_{\ga,\la}\fH\big)
\subseteq
[(X_r\rtimes_{\al,\la}\fH)\cdot(I_{r^{-1}(r\vee F)}\rtimes_{\al,\la}\fH)]
\subseteq
(X_r\rtimes_{\al,\la}\fH) I^\rtimes_{r^{-1}(r\vee F)}
\foral r\in P.
\]
Since $X_F$ is a closed $\ga$-invariant submodule of $\F X$, by Proposition~\ref{P:modsub}~(i) and~(ii) we regard
\[
X_F\rtimes_{\ga,\la}\fH=\sum^\oplus_{r\in P}(X_r I_{r^{-1}(r\vee F)})\rtimes_{\ga,\la}\fH
\]
as a closed submodule of $\F X\rtimes_{\ga,\la}\fH$, and thus
\[
U(X_F\rtimes_{\ga,\la}\fH)\subseteq(X\rtimes_{\al,\la}\fH)_F.
\]

Finally, we compare the representations.
Consider the injective $*$-homomorphism  $\La=U\ka U^*$ from Lemma~\ref{L:Fockcp}~(ii), where $\ka\colon\T_\la(X)\rtimes_{\al,\la}\fH\to\L(\F X\rtimes_{\ga,\la}\fH)$ is the faithful $*$-representation given by Proposition~\ref{P:modrep} for $\D=\T_\la(X)$.
We claim that for every $d\in[\T_\la(X)]_{e_P}\rtimes_{\al,\la}\fH$ the submodule $X_F\rtimes_{\ga,\la}\fH$ is invariant under $\ka(d)$ and
\begin{equation}\label{Eq:kaF}
\ka(d)|_{X_F\rtimes_{\ga,\la}\fH}=\ka_F((\pi_F\rtimes\fH)(d)).
\end{equation}
Indeed, let $f$ be in $C_c(\fH,[\T_\la(X)]_{e_P})$ and $\ze$ be in $C_c(\fH,X_F)$.
Recall that $\ga_\fg(X_F)=X_F$ for all $\fg\in\fH$, and every $f(\fg)$ leaves $X_F$ invariant with $f(\fg)|_{X_F}=\pi_F(f(\fg))$. By the definition of $\ka$ and of $\ka_F$ (defined in \eqref{Eq:kappaf}) we obtain
\[
(\ka(f)\ze)(\fg')
=\int_{\fH}f(\fg)\ga_\fg(\ze(\fg^{-1}\fg'))\,d\mu(\fg)
=\int_{\fH}\pi_F(f(\fg))\ga_\fg(\ze(\fg^{-1}\fg'))\,d\mu(\fg)
=\big(\ka_F(\pi_F\circ f)\ze\big)(\fg').
\]
Since $\pi_F\circ f=(\pi_F\rtimes\fH)(f)$, the claim follows by continuity and density.

Now let $\ze$ be in $X_F\rtimes_{\ga,\la}\fH$.
By \eqref{Eq:kaF} and Lemma~\ref{L:Fockcp}~(ii) we have
\[
U\ka_F\big((\pi_F\rtimes\fH)(\io^\rtimes_\ast(b))\big)\ze
=U\ka(\io^\rtimes_\ast(b))\ze
=\La(\io^\rtimes_\ast(b))U\ze
=\la^\rtimes_\ast(b)U\ze.
\]
As $U\ze\in(X\rtimes_{\al,\la}\fH)_F$, and $\la^\rtimes_\ast(b)\in[\T_\la(X\rtimes_{\al,\la}\fH)]_{e_P}$ leaves $(X\rtimes_{\al,\la}\fH)_F$ invariant and its restriction there is equal to $\pi^\rtimes_F(\la^\rtimes_\ast(b))$, and as $U$ is isometric, we get
\begin{align*}
\big\|\ka_F\big((\pi_F\rtimes\fH)(\io^\rtimes_\ast(b))\big)\ze\big\|
&=
\big\|U\ka_F\big((\pi_F\rtimes\fH)(\io^\rtimes_\ast(b))\big)\ze\big\|\\
&=
\|\la^\rtimes_\ast(b)U\ze\|
=
\|\pi^\rtimes_F(\la^\rtimes_\ast(b))U\ze\|
\leq
\|\pi^\rtimes_F(\la^\rtimes_\ast(b))\|\,\|\ze\|.
\end{align*}
Since $\ka_F$ is faithful, we conclude that
\[
\|(\pi_F\rtimes\fH)(\io^\rtimes_\ast(b))\|
=\big\|\ka_F\big((\pi_F\rtimes\fH)(\io^\rtimes_\ast(b))\big)\big\|
\leq
\|\pi^\rtimes_F(\la^\rtimes_\ast(b))\|,
\]
as required.
\end{proof}

We can now show that the identity representation becomes strong covariant after composing with the canonical quotient map.

\begin{proposition}\label{P:scov}
With the aforementioned notation, the representation
\[
(q_{{\rm sc},\la}\rtimes\fH)\circ\io^\rtimes\colon X\rtimes_{\al,\la}\fH\to(A\times_{X,\la}P)\rtimes_{\dot\al,\la}\fH
\]
is a strong covariant representation of $X\rtimes_{\al,\la}\fH$.
\end{proposition}

\begin{proof}
Note that we have 
\[
((q_{{\rm sc},\la}\rtimes\fH)\circ\io^\rtimes)_\ast=(q_{{\rm sc},\la}\rtimes\fH)\circ\io^\rtimes_\ast,
\]
since both sides are $*$-homomorphisms that agree on the generators of $\T(X\rtimes_{\al,\la}\fH)$.
Now let $b$ be in $\I^\rtimes_{\scv,e}$.
By Lemma~\ref{L:compare} we have $\io^\rtimes_\ast(b)\in[\T_\la(X)]_{e_P}\rtimes_{\al,\la}\fH$. For every finite $F\subseteq G$, applying Lemma~\ref{L:factor} and Lemma~\ref{L:compare} yields
\[
\|((q_{{\rm sc},\la}\rtimes\fH)\circ\io^\rtimes)_\ast(b)\|
=
\|(q_{{\rm sc},\la}\rtimes\fH)(\io^\rtimes_\ast(b))\|
\leq
\|(\pi_F\rtimes\fH)(\io^\rtimes_\ast(b))\|
\leq
\|\pi^\rtimes_F(\la^\rtimes_\ast(b))\|.
\]
By the definition of $\I^\rtimes_{\scv,e}$ we have $\lim_F\|\pi^\rtimes_F(\la^\rtimes_\ast(b))\|=0$, and hence $((q_{{\rm sc},\la}\rtimes\fH)\circ\io^\rtimes)_\ast(b)=0$.
Thus $((q_{{\rm sc},\la}\rtimes\fH)\circ\io^\rtimes)_\ast$ vanishes on $\I^\rtimes_{\scv,e}$ and hence $(q_{{\rm sc},\la}\rtimes\fH)\circ\io^\rtimes$ is strong covariant.
\end{proof}

We will also need the following proposition.

\begin{proposition}\label{P:injnor}
With the aforementioned notation, the representation
\[
(q_{{\rm sc},\la}\rtimes\fH)\circ\io^\rtimes\colon X\rtimes_{\al,\la}\fH\to(A\times_{X,\la}P)\rtimes_{\dot\al,\la}\fH
\]
is an injective representation of $X\rtimes_{\al,\la}\fH$ that admits a normal coaction by $G$, and
\[
\ca((q_{{\rm sc},\la}\rtimes\fH)\circ\io^\rtimes)=(A\times_{X,\la}P)\rtimes_{\dot\al,\la}\fH.
\]
\end{proposition}

\begin{proof}
The restriction of $(q_{{\rm sc},\la}\rtimes\fH)\circ\io^\rtimes$ to the coefficient algebra $A\rtimes_{\al,\la}\fH$ is the restriction of $q_{{\rm sc},\la}\rtimes\fH$ to $A\rtimes_{\al,\la}\fH$.
Recall that $q_{{\rm sc},\la}$ is injective on $\la_{e_P}(A)$, and since $\la_{e_P}(A)$ is $\al$-invariant and 
\[
\dot\al_\fh\circ q_{{\rm sc},\la}=q_{{\rm sc},\la}\circ\al_\fh \foral \fh\in\fH,
\]
the restriction of $q_{{\rm sc},\la}$ to $\la_{e_P}(A)$ is an injective $\fH$-equivariant $*$-homomorphism into $A\times_{X,\la}P$.
Hence it induces an injective $*$-homomorphism
\[
A\rtimes_{\al,\la}\fH\to(A\times_{X,\la}P)\rtimes_{\dot\al,\la}\fH,
\]
which agrees with $q_{{\rm sc},\la}\rtimes\fH$ on $C_c(\fH,\la_{e_P}(A))$, and thus on $A\rtimes_{\al,\la}\fH$ by continuity.
Therefore $(q_{{\rm sc},\la}\rtimes\fH)\circ\io^\rtimes$ is injective. Moreover, since $q_{{\rm sc},\la}$ is surjective, so is $q_{{\rm sc},\la}\rtimes\fH$, and hence Proposition~\ref{P:gen} yields
\begin{align*}
\ca((q_{{\rm sc},\la}\rtimes\fH)\circ\io^\rtimes)
& =
((q_{{\rm sc},\la}\rtimes\fH)\circ\io^\rtimes)_\ast(\T(X\rtimes_{\al,\la}\fH))\\
&=
(q_{{\rm sc},\la}\rtimes\fH)\circ\io^\rtimes_\ast(\T(X\rtimes_{\al,\la}\fH))\\
&=
(q_{{\rm sc},\la}\rtimes\fH)(\ca(\io^\rtimes))\\
& =
(q_{{\rm sc},\la}\rtimes\fH)(\T_\la(X)\rtimes_{\al,\la}\fH)
=
(A\times_{X,\la}P)\rtimes_{\dot\al,\la}\fH.
\end{align*}

We now show that $(q_{{\rm sc},\la}\rtimes\fH)\circ\io^\rtimes$ admits a normal coaction by $G$.
Since $A\times_{X,\la}P$ is the reduced C*-algebra of $\S\C X$, it admits a normal coaction
\[
\de\colon A\times_{X,\la}P\to(A\times_{X,\la}P)\otimes\ca(G);\ q_{{\rm sc},\la}(\la_p(\xi_p))\mapsto q_{{\rm sc},\la}(\la_p(\xi_p))\otimes u_p.
\]
Since 
\[
\dot\al_\fh(q_{{\rm sc},\la}(\la_p(X_p)))=q_{{\rm sc},\la}(\la_p(X_p)) \foral p\in P \text{ and }\fh\in\fH,
\]
the $*$-homomorphisms $\de\circ\dot\al_\fh$ and $(\dot\al_\fh\otimes\id)\circ\de$ agree on the generators of $A\times_{X,\la}P$, and hence $\de$ is $\fH$-equivariant.
Thus there is an induced $*$-homomorphism
\[
\de\rtimes\fH\colon(A\times_{X,\la}P)\rtimes_{\dot\al,\la}\fH\to\big((A\times_{X,\la}P)\otimes\ca(G)\big)\rtimes_{\dot\al\otimes\id,\la}\fH.
\]
By \cite[Lemma~7.16]{Wil07} we have a canonical $*$-isomorphism
\[
\pi \colon \big((A\times_{X,\la}P)\otimes\ca(G)\big)\rtimes_{\dot\al\otimes\id,\la}\fH \to \big((A\times_{X,\la}P)\rtimes_{\dot\al,\la}\fH\big)\otimes\ca(G),
\]
and hence by composing we obtain a $*$-homomorphism
\[
\de':= \pi \circ (\de\rtimes \fH) \colon (A\times_{X,\la}P)\rtimes_{\dot\al,\la}\fH\to\big((A\times_{X,\la}P)\rtimes_{\dot\al,\la}\fH\big)\otimes\ca(G)
\]
such that 
\[
\de'((q_{{\rm sc},\la}\rtimes\fH)(y))=(q_{{\rm sc},\la}\rtimes\fH)(y)\otimes u_p \foral y\in X_p\rtimes_{\al,\la}\fH \text{ and }p\in P.
\]
Therefore $(q_{{\rm sc},\la}\rtimes\fH)\circ\io^\rtimes$ is $G$-equivariant.
Moreover, the composition $(\id\otimes\la_G) \circ \de'$ agrees on $C_c(\fH,A\times_{X,\la}P)$, and thus everywhere, with the $*$-homomorphism induced by the $\fH$-equivariant $*$-homomorphism $(\id\otimes\la_G)\circ\de$, followed by the canonical $*$-isomorphism
\[
\big((A\times_{X,\la}P)\otimes\ca_\la(G)\big)\rtimes_{\dot\al\otimes\id,\la}\fH\simeq\big((A\times_{X,\la}P)\rtimes_{\dot\al,\la}\fH\big)\otimes\ca_\la(G)
\]
of \cite[Lemma~7.16]{Wil07}.
As $\de$ is normal, the $*$-homomorphism $(\id\otimes\la_G)\circ\de$ is injective, and hence so is the $*$-homomorphism it induces.
Therefore $(\id\otimes\la_G) \circ \de'$ is injective and hence $(q_{{\rm sc},\la}\rtimes\fH)\circ\io^\rtimes$ admits a normal coaction by $G$.
\end{proof}

We now obtain the positive resolution of the reduced Hao--Ng isomorphism problem for product systems in full generality.

\begin{theorem}\label{T:HaoNg}
Let $X$ be a product system over a unital subsemigroup $P$ of a discrete group $G$, and let $\al$ be a generalised gauge action of a locally compact Hausdorff group $\fH$ on $\T_\la(X)$.
Let $\dot{\al}$ denote the induced action of $\fH$ on $A\times_{X,\la}P$.
Then the representation
\[
(q_{{\rm sc},\la}\rtimes\fH)\circ\io^\rtimes\colon X\rtimes_{\al,\la}\fH\to(A\times_{X,\la}P)\rtimes_{\dot\al,\la}\fH
\]
induces a canonical $*$-isomorphism
\[
(A\rtimes_{\al,\la}\fH)\times_{X\rtimes_{\al,\la}\fH,\la}P\simeq(A\times_{X,\la}P)\rtimes_{\dot\al,\la}\fH.
\]
\end{theorem}

\begin{proof}
By Proposition~\ref{P:scov} and Proposition~\ref{P:injnor} the representation $(q_{{\rm sc},\la}\rtimes\fH)\circ\io^\rtimes$ is an injective strong covariant representation of $X\rtimes_{\al,\la}\fH$ that admits a normal coaction by $G$, and it generates $(A\times_{X,\la}P)\rtimes_{\dot\al,\la}\fH$.
By universality there is an induced $G$-equivariant canonical $*$-epimorphism
\[
\Phi\colon(A\rtimes_{\al,\la}\fH)\times_{X\rtimes_{\al,\la}\fH}P\to\ca((q_{{\rm sc},\la}\rtimes\fH)\circ\io^\rtimes).
\]
Since $(q_{{\rm sc},\la}\rtimes\fH)\circ\io^\rtimes$ is injective, \cite[Theorem~3.10]{Seh19} yields that $\Phi$ is injective on the fixed point algebra $[(A\rtimes_{\al,\la}\fH)\times_{X\rtimes_{\al,\la}\fH}P]_{e_P}$.
In particular, $\Phi$ induces an isomorphism of the Fell bundles $\S\C(X\rtimes_{\al,\la}\fH)$ and $\{[\ca((q_{{\rm sc},\la}\rtimes\fH)\circ\io^\rtimes)]_g\}_{g\in G}$, and thus a canonical $*$-isomorphism between their reduced C*-algebras.
Since $(q_{{\rm sc},\la}\rtimes\fH)\circ\io^\rtimes$ admits a normal coaction by $G$, by \cite[Proposition~3.7]{Exe97} the C*-algebra $\ca((q_{{\rm sc},\la}\rtimes\fH)\circ\io^\rtimes)$ is canonically $*$-isomorphic to the reduced C*-algebra of the Fell bundle $\{[\ca((q_{{\rm sc},\la}\rtimes\fH)\circ\io^\rtimes)]_g\}_{g\in G}$.
We conclude that
\begin{align*}
(A\rtimes_{\al,\la}\fH)\times_{X\rtimes_{\al,\la}\fH,\la}P
& =
\ca_\la(\S\C(X\rtimes_{\al,\la}\fH))\\
& \simeq
\ca_\la\big(\{[\ca((q_{{\rm sc},\la}\rtimes\fH)\circ\io^\rtimes)]_g\}_{g\in G}\big)\\
&\simeq
\ca((q_{{\rm sc},\la}\rtimes\fH)\circ\io^\rtimes)
=
(A\times_{X,\la}P)\rtimes_{\dot\al,\la}\fH,
\end{align*}
by canonical $*$-isomorphisms, as required.
\end{proof}


\begin{thebibliography}{99}

\bibitem{Aba10}
B.\ Abadie,
\textit{Takai duality for crossed products by Hilbert $C^*$-bimodules},
J.\ Operator Theory \textbf{64} (2010), no.\ 1, 19--34.

\bibitem{Arv89}
W.\ B.\ Arveson,
\textit{Continuous analogues of Fock space},
Mem.\ Amer.\ Math.\ Soc.\ \textbf{80} (1989), no.\ 409, iv+66 pp.

\bibitem{Arv11}
W.\ B.\ Arveson,
\textit{The noncommutative Choquet boundary II: hyperrigidity},
Israel J.\ Math.\ \textbf{184} (2011), 349--385.

\bibitem{BKQR15}
E.\ B\'edos, S.\ Kaliszewski, J.\ Quigg and D.\ Robertson,
\textit{A new look at crossed product correspondences and associated
$C^*$-algebras},
J.\ Math.\ Anal.\ Appl.\ \textbf{426} (2015), no.\ 2, 1080--1098.

\bibitem{BL04}
D.\ P.\ Blecher and C.\ Le Merdy,
\textit{Operator algebras and their modules---an operator space approach},
London Mathematical Society Monographs, New Series, vol.\ 30,
Oxford University Press, Oxford (2004).

\bibitem{BLS18}
N.\ Brownlowe, N.\ S.\ Larsen and N.\ Stammeier,
\textit{$C^*$-algebras of algebraic dynamical systems and right LCM semigroups},
Indiana Univ.\ Math.\ J.\ \textbf{67} (2018), no.\ 6, 2453--2486.

\bibitem{Bus10} A.\ Buss,
\textit{Integrability of dual coactions on Fell bundle $C^*$-algebras},
Bull.\ Braz.\ Math.\ Soc.\ (N.S.) \textbf{41} (2010), no.\ 4, 607--641.

\bibitem{BEW18}
A.\ Buss, S.\ Echterhoff and R.\ Willett,
\textit{Exotic crossed products and the Baum--Connes conjecture},
J.\ Reine Angew.\ Math.\ \textbf{740} (2018), 111--159.

\bibitem{BM09} A.\ Buss and R.\ Meyer,
\textit{Square-integrable coactions of locally compact quantum groups},
Rep.\ Math.\ Phys.\ \textbf{63} (2009), no.\ 1, 191--224.

\bibitem{CLSV11}
T.\ M.\ Carlsen, N.\ S.\ Larsen, A.\ Sims and S.\ T.\ Vittadello,
\textit{Co-universal algebras associated to product systems, and
gauge-invariant uniqueness theorems},
Proc.\ Lond.\ Math.\ Soc.\ (3) \textbf{103} (2011), no.\ 4, 563--600.

\bibitem{CDL13}
J.\ Cuntz, C.\ Deninger and M.\ Laca,
\textit{$C^*$-algebras of Toeplitz type associated with algebraic number fields},
Math.\ Ann.\ \textbf{355} (2013), no.\ 4, 1383--1423.

\bibitem{Din91}
H.\ T.\ Dinh,
\textit{Discrete product systems and their $C^*$-algebras},
J.\ Funct.\ Anal.\ \textbf{102} (1991), no.\ 1, 1--34.

\bibitem{DKKLL22}
A.\ Dor-On, E.\ T.\ A.\ Kakariadis, E.\ G.\ Katsoulis, M.\ Laca and X.\ Li,
\textit{$C^*$-envelopes for operator algebras with a coaction and co-universal
$C^*$-algebras for product systems},
Adv.\ Math.\ \textbf{400} (2022), Paper No.\ 108286, 40~pp.

\bibitem{DK20}
A.\ Dor-On and E.\ G.\ Katsoulis,
\textit{Tensor algebras of product systems and their $C^*$-envelopes},
J.\ Funct.\ Anal.\ \textbf{278} (2020), no.\ 7, Paper No.\ 108416, 32 pp.

\bibitem{DT26}
A.\ Dor-On and I.\ Thompson,
\textit{The Hao--Ng isomorphism theorem for reduced crossed products},
preprint, arXiv:2505.00587v4 (2026).

\bibitem{DT26b}
A.\ Dor-On and I.\ Thompson,
\textit{Obstructions to the full Hao--Ng isomorphism},
preprint, arXiv:2608.31050 (2026).

\bibitem{EKQR06}
S.\ Echterhoff, S.\ Kaliszewski, J.\ Quigg and I.\ Raeburn,
\textit{A categorical approach to imprimitivity theorems for $C^*$-dynamical systems},
Mem.\ Amer.\ Math.\ Soc.\ \textbf{180} (2006), no.\ 850, viii+169 pp.

\bibitem{Exe97} R.\ Exel,
\textit{Amenability for Fell bundles},
J.\ Reine Angew.\ Math.\ \textbf{492} (1997), 41--73.

\bibitem{Exe17}
R.\ Exel,
\textit{Partial dynamical systems, Fell bundles and applications},
Mathematical Surveys and Monographs, vol.\ 224,
American Mathematical Society, Providence, RI (2017).

\bibitem{Fow02}
N.\ J.\ Fowler,
\textit{Discrete product systems of Hilbert bimodules},
Pacific J.\ Math.\ \textbf{204} (2002), no.\ 2, 335--375.

\bibitem{FR98}
N.\ J.\ Fowler and I.\ Raeburn,
\textit{Discrete product systems and twisted crossed products by semigroups},
J.\ Funct.\ Anal.\ \textbf{155} (1998), no.\ 1, 171--204.

\bibitem{Ham79}
M.\ Hamana,
\textit{Injective envelopes of operator systems},
Publ.\ Res.\ Inst.\ Math.\ Sci.\ \textbf{15} (1979), no.\ 3, 773--785.

\bibitem{HN08}
G.\ Hao and C.-K.\ Ng,
\textit{Crossed products of $C^*$-correspondences by amenable group actions},
J.\ Math.\ Anal.\ Appl.\ \textbf{345} (2008), no.\ 2, 702--707.

\bibitem{KKLL22}
E.\ T.\ A.\ Kakariadis, E.\ G.\ Katsoulis, M.\ Laca and X.\ Li,
\textit{Boundary quotient $C^*$-algebras of semigroups},
J.\ Lond.\ Math.\ Soc.\ (2) \textbf{105} (2022), no.\ 4, 2136--2166.

\bibitem{KKLL23} E.\ T.\ A.\ Kakariadis, E.\ G.\ Katsoulis, M.\ Laca and X.\ Li,
\textit{Couniversality and controlled maps on product systems over right LCM semigroups},
Anal. PDE \textbf{16} (2023), no.\ 6, 1433--1483.

\bibitem{KP25}
E.\ T.\ A.\ Kakariadis and I.\ A.\ Paraskevas,
\textit{On Fock covariance for product systems and the reduced Hao--Ng
isomorphism problem by discrete actions},
Proc.\ Roy.\ Soc.\ Edinburgh Sect.\ A, First View (2025), 1--53,
doi:10.1017/prm.2025.21.

\bibitem{Kat17}
E.\ G.\ Katsoulis,
\textit{$C^*$-envelopes and the Hao--Ng isomorphism for discrete groups},
Int.\ Math.\ Res.\ Not.\ IMRN \textbf{2017} (2017), no.\ 18, 5751--5768.

\bibitem{Kat20}
E.\ G.\ Katsoulis,
\textit{Product systems of $C^*$-correspondences and Takai duality},
Israel J.\ Math.\ \textbf{240} (2020), no.~1, 223--251.

\bibitem{KK06}
E.\ G.\ Katsoulis and D.\ W.\ Kribs,
\textit{Tensor algebras of $C^*$-correspondences and their $C^*$-envelopes},
J.\ Funct.\ Anal.\ \textbf{234} (2006), no.\ 1, 226--233.

\bibitem{KR19}
E.\ G.\ Katsoulis and C.\ Ramsey,
\textit{Crossed products of operator algebras},
Mem.\ Amer.\ Math.\ Soc.\ \textbf{258} (2019), no.\ 1240, vii+85 pp.

\bibitem{KR21}
E.\ G.\ Katsoulis and C.\ Ramsey,
\textit{The non-selfadjoint approach to the Hao--Ng isomorphism},
Int.\ Math.\ Res.\ Not.\ IMRN \textbf{2021} (2021), no.\ 2, 1160--1197.

\bibitem{Kat03}
T.\ Katsura,
\textit{A construction of $C^*$-algebras from $C^*$-correspondences},
in \textit{Advances in quantum dynamics}, 173--182,
Contemp.\ Math., vol.\ 335,
Amer.\ Math.\ Soc., Providence, RI (2003).

\bibitem{Kat04}
T.\ Katsura,
\textit{On $C^*$-algebras associated with $C^*$-correspondences},
J.\ Funct.\ Anal.\ \textbf{217} (2004), no.\ 2, 366--401.

\bibitem{KP99}
A.\ Kumjian and D.\ Pask,
\textit{$C^*$-algebras of directed graphs and group actions},
Ergodic Theory Dynam.\ Systems \textbf{19} (1999), no.\ 6, 1503--1519.

\bibitem{KL19a}
B.\ K.\ Kwa\'{s}niewski and N.\ S.\ Larsen,
\textit{Nica--Toeplitz algebras associated with right-tensor $C^*$-precategories
over right LCM semigroups},
Internat.\ J.\ Math.\ \textbf{30} (2019), no.\ 2, 1950013, 57 pp.

\bibitem{KL19b}
B.\ K.\ Kwa\'{s}niewski and N.\ S.\ Larsen,
\textit{Nica--Toeplitz algebras associated with product systems over right
LCM semigroups},
J.\ Math.\ Anal.\ Appl.\ \textbf{470} (2019), no.\ 1, 532--570.

\bibitem{LS22}
M.\ Laca and C.\ F.\ Sehnem,
\textit{Toeplitz algebras of semigroups},
Trans.\ Amer.\ Math.\ Soc.\ \textbf{375} (2022), no.\ 10, 7443--7507.

\bibitem{Lan95}
E.\ C.\ Lance,
\textit{Hilbert $C^*$-modules: a toolkit for operator algebraists},
London Mathematical Society Lecture Note Series, vol.\ 210,
Cambridge University Press, Cambridge (1995).

\bibitem{Li12}
X.\ Li,
\textit{Semigroup $C^*$-algebras and amenability of semigroups},
J.\ Funct.\ Anal.\ \textbf{262} (2012), no.\ 10, 4302--4340.

\bibitem{MT05}
V.\ M.\ Manuilov and E.\ V.\ Troitsky,
\textit{Hilbert $C^*$-modules},
Translations of Mathematical Monographs, vol.\ 226,
American Mathematical Society, Providence, RI (2005).

\bibitem{MS26}
M.\ Mukohara and Y.\ Suzuki,
\textit{Stabilization theorem and symmetric structure of Cuntz--Pimsner algebras},
preprint, arXiv:2605.21128v3 (2026).

\bibitem{Nic92}
A.\ Nica,
\textit{$C^*$-algebras generated by isometries and Wiener--Hopf operators},
J.\ Operator Theory \textbf{27} (1992), no.~1, 17--52.

\bibitem{Nor14}
M.\ D.\ Norling,
\textit{Inverse semigroup $C^*$-algebras associated with left cancellative semigroups},
Proc.\ Edinb.\ Math.\ Soc.\ (2) \textbf{57} (2014), no.\ 2, 533--564.

\bibitem{Pau02}
V.\ I.\ Paulsen,
\textit{Completely bounded maps and operator algebras},
Cambridge Studies in Advanced Mathematics, vol.\ 78,
Cambridge University Press, Cambridge (2002).

\bibitem{Pim97}
M.\ V.\ Pimsner,
\textit{A class of $C^*$-algebras generalizing both Cuntz--Krieger algebras and
crossed products by $\mathbb{Z}$},
in \textit{Free probability theory} (Waterloo, ON, 1995), 189--212,
Fields Inst.\ Commun., vol.\ 12,
Amer.\ Math.\ Soc., Providence, RI (1997).

\bibitem{Qui96}
J.\ Quigg,
\textit{Discrete $C^*$-coactions and $C^*$-algebraic bundles},
J.\ Austral.\ Math.\ Soc.\ Ser.\ A \textbf{60} (1996), no.\ 2, 204--221.

\bibitem{Seh19}
C.\ F.\ Sehnem,
\textit{On $C^*$-algebras associated to product systems},
J.\ Funct.\ Anal.\ \textbf{277} (2019), no.\ 2, 558--593.

\bibitem{Seh22}
C.\ F.\ Sehnem,
\textit{$C^*$-envelopes of tensor algebras of product systems},
J.\ Funct.\ Anal.\ \textbf{283} (2022), no.\ 12, Paper No.~109707, 31 pp.

\bibitem{Wil07}
D.\ P.\ Williams,
\textit{Crossed products of $C^*$-algebras},
Mathematical Surveys and Monographs, vol.\ 134,
American Mathematical Society, Providence, RI (2007).

\end{thebibliography}
\end{document}